\documentclass[11pt]{amsart}

\usepackage[utf8]{inputenc}
\usepackage{amsmath,amssymb,amsthm,mathrsfs}
\usepackage{mathtools}
\usepackage{enumerate}
\usepackage{centernot}

\usepackage{xcolor}
\newif\ifshowchanges
\showchangestrue
\DeclareRobustCommand{\rev}[1]{{\ifshowchanges\fi #1}}
\newenvironment{revision}{\begingroup\ifshowchanges\fi}{\endgroup}

\numberwithin{equation}{section}

\newcommand{\R}{\mathbb{R}}

\theoremstyle{plain}
\newtheorem{theorem}{Theorem}[section]
\newtheorem{corollary}[theorem]{Corollary}
\newtheorem{lemma}[theorem]{Lemma}
\newtheorem{proposition}[theorem]{Proposition}
\newtheorem{conjecture}[theorem]{Conjecture}

\theoremstyle{definition}
\newtheorem{definition}[theorem]{Definition}

\newtheorem{remark}[theorem]{Remark}

\usepackage{amsrefs}
\usepackage[hidelinks]{hyperref}
\hypersetup{linktoc=all,bookmarksdepth=2}
\title[Proof of the Penrose Conjecture 
for Maximally Parameterized Initial Data Sets]
{Proof of the Penrose Conjecture\\
for Maximally Parameterized Initial Data Sets}

\author[Bray, Dolmen, Khuri, Schmidt]
{Hubert Bray, Tuan Dolmen, Marcus Khuri, and David Schmidt}

\address{(Hubert Bray, Tuan Dolmen, and David Schmidt)
Department of Mathematics, Duke University, Durham, NC 27708, USA}
\email{hubert.bray@duke.edu, tuan.dolmen@duke.edu, david.schmidt@duke.edu}

\address{(Marcus Khuri) Department of Mathematics, Stony Brook University,
Stony Brook, NY 11794, USA}
\email{marcus.khuri@stonybrook.edu}

\thanks{M. Khuri acknowledges the support of NSF Grant DMS-2405045.}

\begin{document}

\begin{abstract}
\rev{We prove the spacetime Penrose inequality for asymptotically flat initial data sets $(M^n,g,k)$ in all dimensions $n\geq 3$, under the
dominant energy condition, for data admitting the `maximally parameterized form' defined in equation \eqref{eq:kMP}.} In this case, the second fundamental form $k$ is locally parameterized by $n(n+1)/2$ functions, giving the same formal local
functional degrees of freedom as an arbitrary symmetric $2$-tensor.

\rev{The proof is based on a new 1-parameter family of generalized Schoen--Yau identities and their
gauge-invariant specialization},
inspired by the generalized Jang framework of Bray--Khuri \rev{\cite{BrayKhuri2010,BrayKhuri2011}}.  For \rev{maximally parameterized data}, these identities produce a Riemannian comparison
metric
\[
    \bar g=g+\beta\otimes\beta
\]
with nonnegative scalar curvature $R_{\bar{g}} \geq 0$.  Under a natural nondegeneracy
condition \eqref{eq:simplezero} at the apparent horizon, the comparison metric admits a \rev{smooth
minimal completion up to $\partial M^n$.} Together with suitable asymptotic
control, this reduces the Penrose inequality to its Riemannian
counterpart.
\rev{The all-dimensional conclusion uses the Riemannian Penrose theorem
of Bi--Zhu \cite{BiZhu2026}, together with a minimizing enclosure argument
that allows singular boundaries.} 

More generally, in the spirit of \cite{BrayKhuri2010,BrayKhuri2011}, we note that the Penrose conjecture follows whenever a certain system of PDEs \eqref{eq:gauge-system} involving $k$ can be solved, in which case our new identity \eqref{eq:main-identity} immediately implies $R_{\bar{g}} \geq 0$. A natural problem is thus to regularize our degenerate system to prove a general existence theory, thereby proving the Penrose conjecture. 
Essentially, we believe this is what the OpenAI paper \cite{OpenAIPenrose} posted last week on October 6, 2026 claims to have done.  
\end{abstract}

\maketitle

\tableofcontents

\section{Introduction}

The Penrose inequality \cite{Penrose1973} in mathematical general relativity conjectures that asymptotically flat initial data sets for the Einstein field equations, satisfying the dominant energy condition, must have ADM mass bounded below by a precise function of the total area of the black holes it contains.

The physical heuristic behind the conjecture combines several basic
principles of black hole dynamics.  The weak cosmic censorship
conjecture \cite{Penrose1969} predicts that singularities produced by gravitational
collapse are hidden behind an event horizon.  The Hawking area theorem \cite{Hawking1971},
assuming the appropriate energy condition, implies that the area of
the event horizon is nondecreasing toward the future.  One further
expects an isolated black hole spacetime to radiate away its
nonstationary degrees of freedom and approach a stationary black hole;
in four-dimensional vacuum general relativity, the black hole
uniqueness theorems \cite{ChruscielCostaHeusler2012} identify the stationary end state with a Kerr
black hole.  Since energy may be radiated to null infinity, the mass of
the final black hole is no greater than the initial ADM mass, while its
horizon area is no smaller than the relevant initial horizon area.
Combining these facts with the Kerr mass--area relation leads to the
Penrose inequality. In its initial data formulation, the conjecture takes the following form.

\begin{revision}
\begin{conjecture}[Penrose \cite{Penrose1973}]\label{conj:penrose}
Let $(M^n,g,k)$, $n\geq3$, be a complete asymptotically flat initial data
set satisfying the dominant energy condition $\mu\geq|J|_g$, with compact outermost apparent horizon
boundary $\partial M^n$. Then
\begin{equation}\label{eq:penrose}
 m_{\rm ADM}\geq\frac12
 \left(\frac{\mathcal A_g}{\omega_{n-1}}\right)^{\frac{n-2}{n-1}},
\end{equation}
where $\omega_{n-1}=|\mathbb S^{n-1}|$, $m_{\rm ADM}$ is the
ADM mass, and $\mathcal A_g$ is the
infimum of the areas of smooth compact hypersurfaces enclosing
$\partial M^n$. Moreover, equality in \eqref{eq:penrose} holds if and only if $(M^n,g,k)$ embeds isometrically into the Schwarzschild spacetime with induced metric $g$ and second fundamental form $k$. 
\end{conjecture}
\end{revision}

\rev{The Penrose inequality may be viewed as a refinement of the \textit{positive mass theorem} (PMT) \cite{BrendleWang1,BrendleWang2,HKLZ,SchoenYau1979,SchoenYau1981,Witten1981},
which establishes the energy-linear momentum inequality $E_{\rm ADM}\geq |P_{\rm ADM}|$. In dimension three, equality characterizes spacelike slices of
Minkowski space; in higher dimensions the rigidity statement depends
on the asymptotic hypotheses, as discussed in \cite[Theorem~1.2 and
Remark~1.3]{HKLZ}. The Penrose inequality strengthens this to a sharp lower bound
determined by the black hole geometry. Note that using the area of an outermost apparent horizon instead of the minimal enclosing area is not valid in general, as there are counterexamples \cite{BenDov2004}.}

In the time-symmetric case $k = 0$, the dominant energy condition reduces to
nonnegative scalar curvature, the apparent horizon reduces to a minimal surface,
and the conjecture becomes the \emph{Riemannian Penrose inequality} (RPI).
This was proven by Huisken--Ilmanen \cite{HuiskenIlmanen2001}
for $n = 3$ with a connected boundary 
using inverse mean curvature flow,
and by Bray \cite{Bray2001} for $n = 3$ without the connectedness assumption
using a conformal flow of metrics. 
\rev{Later Bray--Lee \cite{BrayLee2009} extended Bray's result to dimensions $3 \leq n \leq 7$, and more recently Bi--Zhu \cite{BiZhu2026} extended the same approach to all dimensions $n\geq 3$.} Recent progress on the spacetime Penrose inequality includes \cite{AllenBrydenKazarasKhuri2025,AnHe,Dong1,Dong2,Ellithy,Hirsch,KhuriKunduri}.

In the following, we prove a Penrose inequality for a distinguished
class of initial data sets, which we call \emph{maximally parameterized
initial data}.  To each such initial data set, satisfying the dominant energy condition, 
we associate a Riemannian comparison metric
\[
\bar g=g+\beta\otimes\beta
\]
with nonnegative scalar curvature, where $\beta$ is a 1-form.  The remaining issues associated with reducing
the problem to the Riemannian Penrose inequality involve controlling the
behavior of $\bar g$ at the apparent horizon and at spatial infinity.
\rev{Under a natural nondegeneracy condition at the horizon, the comparison metric admits a smooth completion with a totally geodesic, hence minimal, inner boundary};
under suitable asymptotic control, its ADM energy agrees with that 
of the original initial data.  To ensure $R_{\bar{g}} \geq 0$, we generalize the approach taken by Bray--Khuri \cite{BrayKhuri2010,BrayKhuri2011}, which is already a generalization of the Schoen--Yau approach \cite{SchoenYau1981} reducing the spacetime PMT to the Riemannian PMT. \rev{In particular, we present a 1-parameter family of Schoen--Yau identities, parameterized by $t\in [0,1]$, that connects the identity established by Bray--Khuri at $t=0$ to its
gauge-invariant specialization at $t=1$. The latter identity ultimately results in a gauge-invariant
Jang zero-divergence system.} 
\begin{revision}
We will utilize the following naming conventions associated with different values of the parameter:
\begin{center}
\begin{tabular}{@{}lll@{}}
$t=0$ & Generalized Schoen--Yau identity & \eqref{eq:BK}\\[3pt]
$t\in[0,1]$ & New Schoen--Yau identity & \eqref{eq:GSY-param-t}\\[3pt]
$t=1$ & Gauge-invariant Schoen--Yau identity & \eqref{eq:main-identity}
\end{tabular}
\end{center}

The guiding observation is that, for fixed $g$ and $k$, the left-hand
side of the generalized Schoen--Yau identity (the scalar curvature of $\bar{g}$) depends on the graph
variables only through $\beta=\phi\,df$, since
$\bar g=g+\beta\otimes\beta$. We therefore seek a reformulation in
which \emph{each term} on the right-hand side has the same invariance.
This leads to the new Schoen--Yau identity for general $t$, with the
desired termwise invariance at $t=1$ on $\{\beta\neq0\}$; at zeros of $\beta$ we
use the specified smooth intrinsic pair $(\nu,b)$ as explained in
Section~\ref{sec:new-identities}. Under the dominant energy condition,
the $t=1$ identity gives $R_{\bar g}\geq0$ when
$\mathrm{I}=0$ and $\mathrm{II}=0$, the boxed system
\eqref{eq:gauge-system}. Any prescribed $k$ admitting such a solution,
with the global completion and asymptotic properties required in
Theorem~\ref{thm:main}, consequently satisfies the stated ADM-energy
Penrose bound. The maximally parameterized form \eqref{eq:kMP}, with $Q$ satisfying
its defining weighted divergence equation, constructs tensors $k$
for which this system holds. The bound follows for the resulting
initial data under those geometric hypotheses. 
\end{revision}

The following simplified statement records the main consequence of the
construction. A more precise version, \rev{including explicit
boundary-completion hypotheses}, is given in Theorem~\ref{thm:main}.

\begin{revision}
\begin{theorem}[Main theorem]\label{thm:main-intro}
Let $(M^n,g,k)$, $n\geq3$, be an asymptotically flat maximally
parameterized initial data set satisfying the dominant energy condition $\mu\geq|J|_g$,
with compact marginally outer trapped surface (MOTS) boundary $\partial M^n$.
Assume that the associated comparison metric $\bar g$ admits a
smooth completion up to $\partial M^n$ with minimal boundary and
induced metric $g|_{T\partial M^n}$, is asymptotically flat, and
has the same ADM energy as $g$. Then 
\[
E_{\rm ADM}\geq
\frac12\left(
\frac{\mathcal A_g}{\omega_{n-1}}
\right)^{\frac{n-2}{n-1}},
\]
where $\mathcal A_g$ is the infimum of the $g$-areas of smooth compact hypersurfaces enclosing
$\partial M^n$.
\end{theorem}

\begin{remark}[Energy and mass]
ADM energy agrees with ADM mass when $P_{\rm ADM}=0$.
For $E_{\rm ADM}>|P_{\rm ADM}|$, an asymptotic rest-frame boost gives the mass bound if the
boost and any required approximation preserve the theorem's hypotheses
and the minimum-enclosing-area lower bound; cf.\
\cite[Section~6.3]{EichmairHuangLeeSchoen2016} and
\cite[Section~5]{AllenBrydenKazarasKhuri2025}.
\end{remark}
\end{revision}

Two additional issues remain outside the scope of the theorem.  First,
\rev{the completion hypotheses exclude horizon ends for which the}
comparison metric develops more degenerate, potentially cylindrical,
asymptotics; these are discussed in Section~\ref{sec:degenerate}.
Second, the phrase `maximally parameterized' refers to the local
functional degree of freedom count of the construction.  We do not
claim that every global symmetric $2$-tensor $k$ admits such a
presentation; the corresponding global surjectivity problem is
discussed in Section~\ref{sec:surjectivity}.

\subsection*{Acknowledgements}
The authors would like to thank Conghan Dong, Sven Hirsch, and Demetre Kazaras for helpful discussions.

\section{Geometric Preliminaries}

\subsection{Initial data, horizons, and the Riemannian Penrose inequality}

Let $(M^n, g)$ be a smooth, connected, oriented, complete Riemannian manifold
with boundary $\partial M^n$, where $n \geq 3$.
An \emph{initial data set} $(M^n, g, k)$ consists of $(M^n, g)$
together with a smooth symmetric $2$-tensor $k$. The tensor $k$ represents the second fundamental form associated with an isometric embedding of $(M^n,g)$, as a spacelike slice, into an $n+1$-dimensional spacetime. \rev{Throughout this work the dimension superscript $n$ will typically be removed from the manifold notation. We write $M^\circ=M\setminus\partial M$ for the interior.} 

The \emph{energy density} $\mu$ and \emph{momentum density} $J$ are defined by
the constraint equations
\begin{align}
  16\pi\mu &= R_g + (\operatorname{tr}_g k)^2 - |k|_g^2, \label{eq:hamiltonian}\\
  8\pi J &= \operatorname{div}_g k - d(\operatorname{tr}_g k). \label{eq:momentum}
\end{align}
The \emph{dominant energy condition} (DEC) is the requirement
$\mu \geq |J|_g$ at every point of $M$.
An initial data set $(M^n, g, k)$ \rev{will be referred to as} \emph{asymptotically flat} if
there exists a compact set $K \subset M$ and a \rev{diffeomorphism} $\varphi: M\setminus K\to \R^n\setminus B_1(0)$ such that in the associated Cartesian coordinates $x$ the metric and second fundamental form satisfy
\begin{equation}\label{eq:AF-decay}
  \rev{g_{ij}-\delta_{ij}=O_2(r^{-p}),\qquad k_{ij}=O_1(r^{-p-1}),\qquad r=|x|,}
\end{equation}
for some $p > (n-2)/2$. \rev{Here $O_j(r^{-a})$ includes coordinate
derivatives through order $j$, with the corresponding additional powers
of $r^{-1}$. Furthermore, we will assume that $\mu, |J|_g \in L^1(M)$.
For a Riemannian comparison metric, asymptotic flatness includes
the same $O_2$ metric decay and integrable scalar curvature.}
The \emph{ADM energy} and \emph{ADM linear momentum} of an asymptotically flat initial data set are defined by
\begin{equation}\label{eq:ADMEnergy}
E_{\rm ADM} \coloneq
\lim_{r\to\infty}
\frac{1}{2(n-1)\omega_{n-1}}
\int_{S_r}
\sum_{i,j=1}^{n}\bigl(\partial_j g_{ij}-\partial_i g_{jj}\bigr)\upsilon^i\,dA,
\end{equation}
\begin{equation}\label{eq:ADMMomentum}
(P_{\rm ADM})_i\coloneq
\lim_{r\to\infty}
\frac{1}{(n-1)\omega_{n-1}}
\int_{S_r}
\sum_{j=1}^{n}\bigl(k_{ij}-(\operatorname{tr}_g k)g_{ij}\bigr)\upsilon^j\,dA,
\end{equation}
where $S_r$ is the coordinate sphere of radius $r$, and
$\upsilon$ is the unit outer normal. \rev{The stated decay and integrability
conditions give well-defined ADM quantities,
see \cite{Bartnik1986,Chrusciel1986}.}  
For initial data satisfying the hypotheses of the positive mass theorem,
the \rev{ADM energy--momentum vector is future causal}, so that
\[
E_{\rm ADM}\geq |P_{\rm ADM}|.
\]
The corresponding ADM mass is then
\begin{equation}\label{eq:ADMMass}
m_{\rm ADM}
\coloneq
\sqrt{E_{\rm ADM}^2-|P_{\rm ADM}|^2}.
\end{equation}

A smooth, closed, two-sided hypersurface $\Sigma^{n-1} \subset (M, g, k)$
with outward-pointing unit normal $\nu$ is called
a \emph{marginally outer trapped surface} (MOTS) if
$H_\Sigma + \operatorname{tr}_\Sigma k = 0$,
and a \emph{marginally inner trapped surface} (MITS) if
$H_\Sigma - \operatorname{tr}_\Sigma k = 0$,
where $H_\Sigma$ is the mean curvature of $\Sigma$ in $(M, g)$
(positive for a round sphere in flat $\mathbb{R}^n$)
and $\operatorname{tr}_\Sigma k = \operatorname{tr}_g k - k(\nu, \nu)$
is the trace of $k$ restricted to $\Sigma$.
In the time-symmetric case $k = 0$, both MOTS and MITS reduce to
minimal surfaces ($H_\Sigma = 0$), and the Penrose conjecture reduces to
the Riemannian Penrose inequality.
We will apply the following result in the final step of our proof.

\begin{revision}
\begin{theorem}[Riemannian Penrose inequality;
Bray \cite{Bray2001}, Bray--Lee \cite{BrayLee2009},
Bi--Zhu \cite{BiZhu2026}]\label{thm:RPI}
Let $(M^n,g)$, $n\geq3$, be asymptotically flat with $R_g\geq0$,
complete up to a compact almost-minimizing Caccioppoli boundary $\Sigma$
whose regular part is minimal. Assume that $\Sigma$ is its own
outermost $g$-minimizing enclosure and that $g$ admits a
$C^{2,\alpha}$ Riemannian extension across $\Sigma$, for some
$0<\alpha<1$, as in \cite[Theorem~1.2]{BiZhu2026}.
Then
\[
 E_{\rm ADM}\geq\frac12
 \left(\frac{\mathcal H_g^{n-1}(\Sigma)}{\omega_{n-1}}\right)^{\frac{n-2}{n-1}},
\]
where $\mathcal{H}_g^{n-1}$ denotes the $(n-1)$-dimensional Hausdorff measure induced by $g$.
Equality holds if and only if $(M^n,g)$ is isometric to a Riemannian
Schwarzschild exterior.
\end{theorem}
\end{revision}

\subsection{Integrable 1-forms and Foliation Geometry}

Let $U\subset M^\circ$ be an open set on which a smooth unit $1$-form
$\nu$ satisfies
\[
|\nu|_g=1,
\qquad
\nu\wedge d\nu=0.
\]
By the Frobenius theorem, the distribution $\ker\nu$ is locally tangent
to a codimension-one foliation.  Thus, after possibly shrinking $U$,
there exists a smooth local defining function $r$ with $dr\neq0$ such that
\[
\nu=\frac{dr}{|dr|_g}.
\]
We write
\[
\lambda=|dr|_g,
\qquad
\Sigma_r=\{r=\mathrm{const}\}
\]
for the associated local foliation.

Associated to $\nu$ and its level-set foliation we define
the \emph{leafwise metric} $\gamma_{ij} = g_{ij} - \nu_i \nu_j$,
the \emph{geodesic curvature} $\kappa^i = \nu^j \nabla_j \nu^i$
of the $\nu$-integral curves (which satisfies $g(\kappa, \nu) = 0$),
the \emph{mean curvature} $H_\Sigma = \operatorname{div}_g(\nu)$
of the leaves,
\rev{and the} \emph{second fundamental form}
\begin{equation}\label{eq:P}
  P_{ij}
  =
  \tfrac12(\nabla_i\nu_j+\nabla_j\nu_i)
  -\tfrac12(\nu_i\kappa_j+\nu_j\kappa_i).
\end{equation}

The tensor $P$ is symmetric, satisfies $P\nu = 0$,
and has trace $\operatorname{tr}_g P = H_\Sigma$. For a smooth function $\psi$ on $M$, we denote its leafwise
differential by
\[
d_\Sigma\psi
\coloneq
d\psi-\nu(\psi)\nu.
\]
Equivalently, $d_\Sigma\psi$ is the restriction of $d\psi$ to the
tangent spaces of the leaves, extended by zero on $\nu$. For a vector field $X$ tangential to the leaves
(i.e.\ $g(X, \nu) = 0$),
the ambient and leaf divergences are related by
\begin{equation}\label{eq:leaf-div}
  \operatorname{div}_\Sigma(X) = \operatorname{div}_g(X) + g(X, \kappa).
\end{equation}

The following identity relates the geodesic curvature to the foliation function $r$, and becomes useful later on. 

\begin{lemma}\label{lem:kappa-log-lambda}
Let $\nu=dr/|dr|_g$ and $\lambda=|dr|_g$. Then
\begin{equation}\label{eq:kappa-log-lambda}
\kappa=d_\Sigma\log\lambda.
\end{equation}
\end{lemma}

\begin{proof}
Write $\nu_i = \partial_i r / \lambda$. Differentiating:
\[
  \nabla_j \nu_i = \frac{\nabla_j \nabla_i r}{\lambda} - \frac{\nu_i\, \partial_j \lambda}{\lambda}.
\]
Contracting with $\nu^j$ and using $\nu^j \nabla_j \nabla_i r = \partial_i \lambda$
(which follows from $\nu^j \nabla_j \nabla_i r = \nu^j \nabla_i \nabla_j r
= \nabla_i(\nu^j \nabla_j r) - \nabla_j r \,\nabla_i \nu^j
= \nabla_i \lambda \rev{-} \lambda \nu^j \nabla_i \nu_j = \nabla_i \lambda$,
since $\nu^j \nabla_i \nu_j = \tfrac{1}{2} \nabla_i |\nu|_g^2 = 0$):
\[
\kappa_i
=
\partial_i\log\lambda
-
\nu_i\,\nu(\log\lambda)
=
(d_\Sigma\log\lambda)_i.\qedhere
\]
\end{proof}

Finally, let $b\colon M\to\mathbb R$ be a smooth scalar function and set
\[
W=(1+b^2)^{1/2},
\qquad
\bar g=g+b^2\,\nu\otimes\nu.
\]
We refer to $\bar g$ as the associated Riemannian comparison metric.
Its inverse metric and volume element are
\[
\bar g^{ij}=g^{ij}-v^iv^j,
\qquad
\sqrt{\det\bar g}=W\sqrt{\det g},
\]
where
\[
v=\frac{b}{W}\nu.
\]
In particular,
\[
|v|_g=\frac{|b|}{W}<1
\]
whenever $b$ is finite.

\section{\rev{The New Schoen--Yau Identity and Its Gauge-Invariant Form}}
\label{sec:new-identities}

\subsection{\rev{The Generalized Schoen--Yau Identity ($t=0$)}}

We review the Jang deformation, originally applied by Schoen--Yau to successfully reduce the spacetime PMT to the Riemannian PMT and later adapted by Bray--Khuri in an attempt to do the same for the Penrose inequality. The proof of Theorem~\ref{thm:main} relies on a modification of the latter strategy, resulting in \rev{the new Schoen--Yau identity and its gauge-invariant $t=1$ specialization, introduced below}. 

Given asymptotically flat initial data $(M,g,k)$ satisfying the DEC, consider smooth functions $f,\phi:M\to \R$ \rev{with $\phi>0$} and construct the deformed metric \begin{equation}
    \bar{g} \coloneq g + \phi^2\,df\otimes df.
\end{equation} Using this metric, construct the warped product spacetime $(\R \times M , -\phi^2\,dt^2 + \bar{g})$. The Riemannian manifold $(M,g)$ embeds as a graph in $(\R \times M , -\phi^2\,dt^2 + \bar{g})$ with metric $g = \bar{g} - \phi^2 \, df \otimes df$ and second fundamental form \begin{equation}
    h = \frac{\phi \nabla^2_g f + d\phi \otimes df + df \otimes d\phi}{\sqrt{1+\phi^2|df|_g^2}}.
\end{equation} Define the difference tensor \begin{equation}
    \delta(0) \coloneq h-k.
\end{equation} The choice of notation here will become clear in the next section when we deform this difference to derive \rev{the new Schoen--Yau identity for general $t$}. 
The scalar curvatures of the metrics $\bar{g}$ and $g$ are related by the generalized Schoen--Yau identity, developed by Bray--Khuri \rev{\cite{BrayKhuri2010,BrayKhuri2011}}.

\begin{theorem}[\rev{Generalized Schoen--Yau identity, $t=0$}]
\label{id:BK}
For any smooth $f$, $\phi > 0$ on $(M, g, k)$, define the difference $\delta(0) \coloneq h - k$ and the vector $v \coloneq \frac{\phi \,\nabla f}{\sqrt{1+\phi^2|df|_g^2}}$. Then the following identity holds for $\bar{g} = g + \phi^2 df\otimes df$:
\begin{revision}
\begin{equation}\label{eq:BK}
\boxed{\begin{gathered}
\text{\bfseries Generalized Schoen--Yau identity}\quad(t=0)\\[5pt]
\begin{aligned}
R_{\bar g}={}&16\pi\bigl(\mu-J(v)\bigr)
+|\delta(0)|_{\bar g}^2
+2|\delta(0)(v,\cdot)|_{\bar g}^2\\
&-\frac{2}{\phi}\operatorname{div}_{\bar g}
\bigl(\phi\,\delta(0)(v,\cdot)\bigr)
+\bigl(\operatorname{tr}_{\bar g}\delta(0)\bigr)^2\\
&+2\bigl(\operatorname{tr}_{\bar g}k+k(v,v)\bigr)
\operatorname{tr}_{\bar g}\delta(0)
+2v\bigl(\operatorname{tr}_{\bar g}\delta(0)\bigr).
\end{aligned}
\end{gathered}}
\end{equation}
\end{revision}
\end{theorem}

The proof is a long calculation that Bray and Khuri carry out in the appendices of their paper. 
Suppose we choose $f$ and $\phi$ to satisfy the \textit{generalized Jang equation} 
\begin{equation}\label{eq:GenJang}
    \text{tr}_{\bar{g}}h = \text{tr}_{\bar{g}}k.
\end{equation}
For such $f$ and $\phi$, the generalized Schoen--Yau identity becomes 
\begin{equation}
    R_{\bar{g}} = 16\pi (\mu - J(v)) + |\delta(0)|^2_{\bar{g}} + 2|\delta(0)(v,\cdot)|^2_{\bar{g}} - \frac{2}{\phi}\text{div}_{\bar{g}}(\phi\, \delta(0)(v,\cdot)).
\end{equation} 
The trick to make this applicable to geometric inequalities is the proper choice of second equation for $f$ and $\phi$. The choice applicable to the PMT is simply $\phi =1$. In this case, the identity simplifies to
\begin{equation}
    \rev{R_{\bar{g}} \geq - 2 \operatorname{div}_{\bar{g}}\!\bigl(\delta(0)(v,\cdot)\bigr)}.
\end{equation}
This is \textit{almost} nonnegative scalar curvature. Schoen--Yau then complete the argument by conformally changing the metric $\bar{g}$ to a metric of zero scalar curvature (to which the Riemannian PMT applies) \rev{and ADM energy no greater than that of $\bar g$}. 

Several complications arise when one tries to apply this approach to the spacetime Penrose inequality. First of all, it is unclear what the right choice for $\phi$ is in this case. One cannot use $\phi =1$ since, in the case of equality, we want $(M,g,k)$ to embed isometrically into the Schwarzschild spacetime, which should require $\rev{\phi = \sqrt{1-\frac{2m}{r^{n-2}}}}$. The most obvious choice would be that $\phi$ is chosen to satisfy \begin{equation}\label{eq:zero-div}
\frac{1}{\phi}
\operatorname{div}_{\bar g}
\bigl(
\phi\,\delta(0)(v,\cdot)
\bigr)
=0.
\end{equation} Equation~\eqref{eq:zero-div}, together with the generalized Jang equation
\eqref{eq:GenJang}, forms the Jang zero-divergence system introduced by
Bray--Khuri \rev{\cite{BrayKhuri2010,BrayKhuri2011}}. \rev{However, Jaracz \cite{Jaracz2023} constructed spherically symmetric initial data sets for which no smooth radial solution of this system exists with positive exterior lapse and the asymptotics required for the Penrose application. His theorem does not rule out nonsymmetric solutions.} 

The key observation that allows us to apply this program to prove Theorem~\ref{thm:main} below, is that the left-hand side of \eqref{eq:BK} depends on the 1-form $\beta = \phi \,df$ rather than the functions $f$ and $\phi$ independently. In particular, changing $\tilde{f}= F(f)$ and $\tilde{\phi} = \frac{\phi}{F'(f)}$ \rev{with $F'>0$} leaves $\beta$ invariant and thus \eqref{eq:BK} still holds. This motivates the introduction of a 1-parameter family of deformed
difference tensors.  The special value $t=1$ will eliminate the
dependence on the individual graph variables $f$ and $\phi$, \rev{on the regular graph set}, and leads to
a gauge-invariant formulation of the Jang zero-divergence system.

\subsection{\rev{The New Schoen--Yau Identity (General $t$)}}

The observation in the previous section allows us to extend the
generalized Schoen--Yau identity of Bray--Khuri to
\rev{the new Schoen--Yau identity, a $1$-parameter family}. 

\rev{Suppose throughout that $f,\phi:M\to\mathbb R$ are smooth functions on an initial data set $(M,g,k)$, with $\phi>0$.} We can define the $1$-form
\[
\beta=\phi\,df
\]
as above and note that $\beta\wedge d\beta=0$; therefore $\beta$ is
integrable. We can normalize
\[
\nu=\frac{\beta}{|\beta|_g}
\]
wherever $\beta\neq0$. \rev{Set
$b=|\beta|_g=\phi|df|_g$ and
$W=\sqrt{1+b^2}=\sqrt{1+\phi^2|df|_g^2}$.} We define
\begin{equation}
    \sigma \coloneq k(\nu,\nu)
    -h(\nu,\nu)
\end{equation}
and the deformed difference tensor
\begin{equation}
    \delta(t)\coloneq h-k+t\sigma\,\nu\otimes\nu.
\end{equation}
Thus $\delta(t)$ deforms the normal--normal component of the
difference $h-k$. \rev{The new Schoen--Yau identity below demonstrates} that we can relate the scalar curvature of the
generalized Jang metric
\[
\bar g=g+\beta\otimes\beta
\]
to the deformed difference tensor $\delta(t)$ and the \rev{matter term
$\mu-J(v)$}.
\rev{The interval $t\in[0,1]$ gives the interpolation described in
the introduction; the identity itself holds for every constant
$t\in\mathbb R$.}

\begin{theorem}[\rev{New Schoen--Yau identity, general $t$}]\label{thm:GSY-t-extended}
\rev{For every constant $t\in\mathbb R$ and any smooth functions $f$ and $\phi>0$ on $(M,g,k)$, the following identity holds on the regular graph set $U=\{df\neq0\}=\{\beta\neq0\}$, where $\nu$, $\sigma$, and $\delta(t)$ are defined:}
\begin{revision}
\begin{equation}\label{eq:GSY-param-t}
\boxed{\begin{gathered}
\text{\bfseries New Schoen--Yau identity}\quad(\text{general }t)\\[5pt]
\begin{aligned}
R_{\bar g}={}&16\pi\bigl(\mu-J(v)\bigr)
+|\delta(t)|_{\bar g}^2
+2|\delta(t)(v,\cdot)|_{\bar g}^2\\
&-\frac{2}{\phi}\operatorname{div}_{\bar g}
\bigl(\phi\,\delta(t)(v,\cdot)\bigr)
+\frac{2\sigma^2t(1-t)}{W^2}\\
&+\bigl(\operatorname{tr}_{\bar g}\delta(t)\bigr)^2
+2\bigl(\operatorname{tr}_{\bar g}k+k(v,v)\bigr)
\operatorname{tr}_{\bar g}\delta(t)\\
&+2v\bigl(\operatorname{tr}_{\bar g}\delta(t)\bigr).
\end{aligned}
\end{gathered}}
\end{equation}
\end{revision}
At $t=0$, the correction term vanishes and
$\delta(0)=h-k$, recovering the generalized Schoen--Yau identity
of Theorem~\ref{id:BK}. 
\end{theorem}

\begin{proof}
The proof is obtained by differentiating the right-hand side of
\eqref{eq:GSY-param-t} with respect to $t$ and showing that it is
independent of $t$. The complete calculation is given in
Appendix~\ref{app:one-parameter-identity}.
\end{proof}

\subsection{\rev{The Gauge-Invariant Schoen--Yau Identity ($t=1$)}}

\begin{revision}
On $\{\beta\neq0\}$, the gauge-invariant Schoen--Yau identity ($t=1$) and its separate terms depend only on $g$, $k$, and $\beta=\phi\,df$, so they are invariant under graph reparametrization. To state the intrinsic identity also at zeros, specify a smooth integrable unit 1-form $\nu$ and a smooth signed function $b$, and put $W=\sqrt{1+b^2}$. Define directly
\begin{equation}\label{eq:Delta-intrinsic}
\Delta\coloneq
\frac{b}{W}P-k+k(\nu,\nu)\nu\otimes\nu
+\frac{1}{W}\bigl(\nu\otimes d_\Sigma b+d_\Sigma b\otimes\nu\bigr),
\end{equation}
where $P$ is the leafwise second fundamental form extended by zero in the normal direction. This tensor is smooth for the specified $(\nu,b)$, including at $b=0$, and agrees with $\delta(1)$ on the regular graph set. At zeros of $\beta$, the separate terms can depend on the chosen foliation; they are not determined by $\beta$ alone.
\end{revision}

\begin{theorem}[\rev{Gauge-invariant Schoen--Yau identity, $t=1$}]
\label{thm:gauge-invariant-new}
    Let $(M, g, k)$ be an initial data set.
Let $\nu$ be a smooth $1$-form satisfying
\[
|\nu|_g=1,
\qquad
\nu\wedge d\nu=0,
\]
and let $\Sigma$ denote the leaves of the codimension-one foliation
tangent to $\ker\nu$.  Denote the mean curvature of the leaves by 
\begin{equation}
    H_{\Sigma} = \operatorname{div}_g(\nu)
\end{equation}  
and the geodesic curvature by  
\begin{equation}
    \kappa = \nabla_{\nu}\nu.
\end{equation}  
Suppose $b$ is a smooth scalar function and define $W = (1 + b^2)^{1/2}$ and $\bar{g} = g + b^2\, \nu \otimes \nu$. \rev{With $\Delta$ defined by \eqref{eq:Delta-intrinsic}, the following identity holds throughout the domain of the smooth pair $(\nu,b)$, including where $b=0$:}
\begin{revision}
\begin{equation}\label{eq:main-identity}
\boxed{\begin{gathered}
\text{\bfseries Gauge-invariant Schoen--Yau identity}\quad(t=1)\\[5pt]
\begin{aligned}
R_{\bar g}={}&16\pi\left(\mu-J\left(\frac bW\nu\right)\right)
+|\Delta|_g^2-\frac2W \rev{\mathrm{II}}\\
&+\rev{\mathrm{I}}^2+2\operatorname{tr}_g k\,\rev{\mathrm{I}}
+\frac{2b}{W}\nu(\rev{\mathrm{I}}).
\end{aligned}
\end{gathered}}
\end{equation}
\end{revision}
where \begin{equation}
    \rev{\mathrm{I}} \coloneq \frac{b}{W}\, H_\Sigma - \operatorname{tr}_\Sigma k
\end{equation} and \begin{equation}
\rev{\mathrm{II}} \coloneq
2\left\langle
d_\Sigma b,
\Delta(\nu,\cdot)|_\Sigma
\right\rangle_g
+
b\left(
\operatorname{div}_\Sigma
\bigl(\Delta(\nu,\cdot)|_\Sigma\bigr)
-
2\Delta(\nu,\kappa)
\right).
\end{equation}  Every term in \eqref{eq:main-identity} depends only on $g$, $k$, $\nu$, and $b$. \end{theorem}

\begin{proof}
\begin{revision}
Where $b\neq0$, specialize the new Schoen--Yau identity to $t=1$ and rewrite it in terms of the intrinsic pair $(\nu,b)$, as in Appendix~\ref{app:gauge-invariant-identity}. At a point $p$ with $b(p)=0$, apply this already established identity to $b_\varepsilon=b+\varepsilon$ on a neighborhood of $p$ on which $b_\varepsilon\neq0$, and let $\varepsilon\to0$. All intrinsic terms, and the scalar curvature of $g+b_\varepsilon^2\nu\otimes\nu$, converge smoothly at $p$. This proves the identity even when the zero set of $b$ has interior.
\end{revision}
\end{proof}

\begin{remark}
Although the expression for $\rev{\mathrm{II}}$ is somewhat lengthy, it only depends
on the mixed normal--tangential component of the discrepancy tensor
$\Delta$.  In fact,
\[
\Delta(\nu,\cdot)|_\Sigma
=
\frac{d_\Sigma b}{W}
-
k(\nu,\cdot)|_\Sigma.
\]
This observation will be useful in the next section: the maximally
parameterized form of $k$ is chosen precisely so that this mixed
component is represented by the tangential $1$-form $Q$, reducing the
condition $\rev{\mathrm{II}}=0$ to a single weighted leafwise divergence equation.
\end{remark}

\rev{The gauge-invariant Schoen--Yau identity motivates the following
\emph{gauge-invariant Jang zero-divergence system}. We use the notation
$\mathrm{I}$ and $\mathrm{II}$ throughout for the two expressions defined
above. The system $\mathrm{I}=0$, $\mathrm{II}=0$ is explicitly}
\begin{revision}
\begin{equation}\label{eq:gauge-system}
\boxed{\begin{gathered}
\text{\bfseries Gauge-invariant Jang zero-divergence system}\\[5pt]
\begin{aligned}
\mathrm{I}&=\frac bW H_\Sigma-\operatorname{tr}_\Sigma k=0,\\[3pt]
\mathrm{II}&=2\left\langle d_\Sigma b,
\Delta(\nu,\cdot)|_\Sigma\right\rangle_g\\
&\quad+b\left(\operatorname{div}_\Sigma
\bigl(\Delta(\nu,\cdot)|_\Sigma\bigr)
-2\Delta(\nu,\kappa)\right)=0.
\end{aligned}
\end{gathered}}
\end{equation}
\end{revision}
We refer to \rev{$\mathrm{I}=0$} as the \emph{trace equation} and to \rev{$\mathrm{II}=0$} as the
\emph{zero-divergence equation}.  When both equations hold, the
\rev{gauge-invariant Schoen--Yau identity} reduces to
\[
R_{\bar g}
=
16\pi\left(
\mu-J\left(\frac{b}{W}\nu\right)
\right)
+
|\Delta|_g^2.
\]
Since
\[
\left|\frac{b}{W}\nu\right|_g<1
\]
where $b$ is finite, the dominant energy condition implies
\[
R_{\bar g}\ge0.
\]
Thus, solving the gauge-invariant Jang zero-divergence equations
produces a Riemannian comparison metric
\[
\bar g=g+b^2\nu\otimes\nu
\]
with nonnegative scalar curvature.

In the following section, we provide an explicit parameterization of $k$ satisfying $\rev{\mathrm{I}} =0$ and $\rev{\mathrm{II}} =0$, motivating the distinguished class of initial data for which the Penrose inequality is satisfied. 

\section{Maximally Parameterized Initial Data}

The equations $\rev{\mathrm{I}}=0$ and $\rev{\mathrm{II}}=0$ constrain only two pieces of $k$
relative to the foliation: its tangential trace and one scalar condition
on its mixed normal--tangential component.  This makes the system
particularly amenable to an explicit parametrization.

Any symmetric $2$-tensor decomposes uniquely as
\begin{equation}\label{eq:kdecomp}
k
=
A\,\nu\otimes\nu
+
B\,\gamma
+
C\otimes\nu+\nu\otimes C
+
\hat{k},
\end{equation}
where $A,B$ are scalar functions, $C$ is tangential, and $\hat{k}$
is symmetric, tangential, and trace-free with respect to $\gamma$.
The trace equation $\rev{\mathrm{I}}=0$ determines the tangential trace
\begin{equation}\label{eq:Bsolution}
B=\frac{bH_\Sigma}{(n-1)W}.
\end{equation}
For the mixed component, recall that
\[
\Delta(\nu,\cdot)|_\Sigma
=
\frac{d_\Sigma b}{W}
-
k(\nu,\cdot)|_\Sigma
=
\frac{d_\Sigma b}{W}-C.
\]
We therefore introduce the tangential $1$-form
\begin{equation}\label{eq:Qdef}
Q
\coloneq
\Delta(\nu,\cdot)|_\Sigma
=
\frac{d_\Sigma b}{W}-C.
\end{equation}
Equivalently,
\begin{equation}\label{eq:Csolution}
C=\frac{d_\Sigma b}{W}-Q.
\end{equation}
With this notation, the zero-divergence equation $\rev{\mathrm{II}}=0$ becomes
\begin{equation}\label{eq:Qeq}
2\langle d_\Sigma b,Q\rangle_g
+
b\bigl(\operatorname{div}_\Sigma Q-2Q(\kappa)\bigr)
=0.
\end{equation} 
These observations motivate the following definition.

\begin{definition}[Maximally Parameterized Presentation]
    Let $(M,g)$ be a complete Riemannian manifold. For an \rev{open subset} $U\subset M^\circ$, a \textbf{maximally parameterized presentation} on $U$ consists of data $(\nu,b,A,Q,\hat{k})$ with the following properties. 
    \begin{enumerate}
        \item $\nu\in \Omega^1(U)$ is a smooth $1$-form satisfying $|\nu|_g=1$ and $\nu\wedge d\nu =0$. 
        \item $b,A\in C^\infty(U)$ are arbitrary smooth, signed scalar functions. 
        \item $Q\in \Omega^1(U)$ is a smooth $1$-form tangent to the \rev{leaves} of the foliation determined by $\nu$: $Q(\nu) =0$. Moreover, $Q$ satisfies the weighted leafwise divergence equation \begin{equation}
            2\langle d_{\Sigma}b,Q\rangle_g + b(\text{div}_\Sigma Q - 2Q(\kappa)) =0. 
        \end{equation}
        \item $\hat{k}\in \Gamma(S^2T^*U)$ is a smooth symmetric 2-tensor that is tangential and tangentially trace free: $\hat{k}(\nu,\cdot) =0$ and $\text{tr}_\Sigma \hat{k}=0$. 
    \end{enumerate} From this data, one can build the symmetric $2$-tensor \begin{equation}\label{eq:kMP}
k_{MP}
=
A\,\nu\otimes\nu
+
\frac{bH_\Sigma}{(n-1)W}\gamma
+
\left(
\frac{d_\Sigma b}{W}-Q
\right)\otimes\nu
+
\nu\otimes
\left(
\frac{d_\Sigma b}{W}-Q
\right)
+
\hat{k},
\end{equation} 
where $\gamma \coloneq g - \nu\otimes \nu$ is the leafwise metric. 
\end{definition} 

The following proposition demonstrates the utility of maximally parameterized presentations, namely, they yield initial data sets $(M,g,k)$ satisfying the \rev{gauge-invariant} Jang zero-divergence equations. 

\begin{proposition}\label{prop:MP-zero-divergence}
Every maximally parameterized presentation satisfies the
gauge-invariant Jang zero-divergence equations
\[
\rev{\mathrm{I}=0,\qquad\mathrm{II}=0.}
\]
\end{proposition}

\begin{proof}
Taking the tangential trace of \eqref{eq:kMP} gives
\[
\operatorname{tr}_\Sigma k_{MP}
=
\frac{bH_\Sigma}{(n-1)W}
\operatorname{tr}_\Sigma\gamma
=
\frac{bH_\Sigma}{W},
\]
since $\operatorname{tr}_\Sigma\gamma=n-1$.  Hence
\[
\rev{\mathrm{I}}
=
\frac{b}{W}H_\Sigma
-
\operatorname{tr}_\Sigma k_{MP}
=
0.
\]

For the mixed normal--tangential component, \eqref{eq:kMP} gives
\[
k_{MP}(\nu,\cdot)|_\Sigma
=
\frac{d_\Sigma b}{W}-Q.
\]
Therefore
\[
\Delta(\nu,\cdot)|_\Sigma
=
\frac{d_\Sigma b}{W}
-
k_{MP}(\nu,\cdot)|_\Sigma
=
Q.
\]
Substituting this into the definition of $\rev{\mathrm{II}}$ yields
\[
\rev{\mathrm{II}}
=
2\langle d_\Sigma b,Q\rangle_g
+
b\bigl(
\operatorname{div}_\Sigma Q
-
2Q(\kappa)
\bigr),
\]
which vanishes by the defining equation for $Q$.
\end{proof}

We have the following immediate corollary which is essential to the proof of the Penrose inequality given in the forthcoming sections. 

\begin{corollary}\label{cor:MP-local-scalar}
Let $(\nu,b,A,Q,\hat{k})$ be a maximally parameterized
presentation on an open set $U\subset M^\circ$, and let
$k=k_{MP}$ on $U$.  If the initial data $(g,k)$ satisfy the dominant
energy condition on $U$, then the comparison metric
\[
\bar g
=
g+b^2\nu\otimes\nu
\]
satisfies
\[
R_{\bar g}\geq0
\qquad\text{on }U.
\]
\end{corollary}

With these results in hand, we can define the class of initial data for which we can prove the Penrose inequality. 

\begin{definition}[Maximally Parameterized Initial Data]\label{def:MP-initial-data}
A smooth initial data set $(M,g,k)$ is \emph{maximally parameterized}
if there exists a closed set
\[
\mathcal{S}\subset M^\circ
\]
with empty interior, and a maximally parameterized presentation
$(\nu,b,A,Q,\hat{k})$
on
$U=M^\circ\setminus \mathcal{S}$
such that:
\begin{enumerate}
    \item the $1$-form
    \[
    \beta\coloneq b\nu
    \]
    extends smoothly across $\mathcal{S}$ to a smooth $1$-form on $M^\circ$;

    \item the tensor determined by the presentation satisfies
    \[
    k_{MP}=k|_U.
    \]
\end{enumerate}
\end{definition}

\begin{remark}
    The set $\mathcal S$ records singularities which may be present in the foliation from the maximally parameterized presentation, not singularities in the initial data. \rev{For multiple boundary components, a foliation by compact leaves
connecting separate near-boundary leaves to connected leaves near
infinity must undergo topology change; singularities of such a
foliation are therefore natural.}  The scalar curvature argument only needs the presentation on a dense regular set, provided the actual geometric fields extend smoothly.
\end{remark}

\begin{remark}
On any regularly foliated neighborhood
$U\subset M^\circ\setminus\mathcal S$ on which $b\neq0$, integrability
implies the existence of a local defining function $r$ such that
\[
\nu=\frac{dr}{|dr|_g}.
\]
Writing $\lambda=|dr|_g$, recall from
Lemma~\ref{lem:kappa-log-lambda} that
\[
\kappa=d_\Sigma\log\lambda.
\]
Taking the divergence along the leaves of the foliation gives
\begin{align*}
\operatorname{div}_\Sigma\left(\frac{b^2}{\lambda^2}Q\right)
&=
\frac{1}{\lambda^2}\operatorname{div}_\Sigma(b^2Q)
-\frac{2b^2}{\lambda^2}Q(\kappa)\\
&=
\frac{1}{\lambda^2}
\left[
2b\langle d_\Sigma b,Q\rangle_g
+b^2\operatorname{div}_\Sigma Q
-2b^2Q(\kappa)
\right]\\
&=
\frac{b}{\lambda^2}
\left[
2\langle d_\Sigma b,Q\rangle_g
+b\bigl(\operatorname{div}_\Sigma Q-2Q(\kappa)\bigr)
\right]\\
&=0.
\end{align*}
Thus
\[
\operatorname{div}_\Sigma
\left(\frac{b^2}{\lambda^2}Q\right)=0.
\]
Therefore, on a contractible open subset of any leaf, the $1$-form
\[
\frac{b^2}{\lambda^2}Q
\]
is co-closed. Equivalently, its Hodge dual is closed. By the
Poincar\'e lemma, there exists an $(n-3)$-form $\Psi$ such that,
up to the sign determined by the Hodge-star convention,
\[
\frac{b^2}{\lambda^2}Q
=
\ast_\Sigma d_\Sigma\Psi.
\]
Hence, wherever $b\neq0$,
\[
Q
=
\frac{\lambda^2}{b^2}\ast_\Sigma d_\Sigma\Psi.
\]
The potential $\Psi$ is determined only up to the addition of a closed
$(n-3)$-form. Locally, for $n\ge4$, this includes the gauge freedom
\[
\Psi\longmapsto\Psi+d_\Sigma\chi.
\]
When $n=3$, $\Psi$ is a function and is determined up to an additive
constant. The covector $Q$ is intrinsic, while $\Psi$ gives a local
parameterization of solutions to the weighted divergence equation.
\end{remark}

\begin{remark}
We can give a formal local functional degree-of-freedom count as follows.
On a regular chart, the foliation function, $b$, and $A$ each contribute
one degree of freedom. The symmetric tangential trace-free tensor
$\hat{k}$ on an $(n-1)$-dimensional leaf contributes
\[
\frac{n(n-1)}{2}-1
\]
degrees of freedom. Finally, the tangential $1$-form $Q$ has $n-1$
components and satisfies one scalar first-order equation, so formally it
contributes $n-2$ degrees of freedom. Thus the total is
\[
1+1+1+
\left(\frac{n(n-1)}{2}-1\right)
+(n-2)
=
\frac{n(n+1)}{2},
\]
which is the number of independent components of a symmetric $2$-tensor
$k$ on an $n$-dimensional manifold.
It is important to note that this is only a formal local count, not a
global surjectivity statement. This count is what motivates the name
`maximally parameterized'.
\end{remark}

\begin{theorem}\label{thm:MP-nonnegative-scalar}
Let $(M,g,k)$ be a maximally parameterized initial data set satisfying
the dominant energy condition. Then
\[
\bar g=g+\beta\otimes\beta
\]
defines a smooth Riemannian metric on $M^\circ$ satisfying
\[
R_{\bar g}\geq0.
\]
\end{theorem}

\begin{proof}
Let $\mathcal S\subset M^\circ$ be the closed singular set appearing in the
definition of a maximally parameterized initial data set, and set
\[
U=M^\circ\setminus \mathcal S.
\]
On $U$, the maximally parameterized presentation satisfies the
gauge-invariant Jang zero-divergence equations. Hence, by
Corollary~\ref{cor:MP-local-scalar},
\[
R_{\bar g}\geq0
\qquad\text{on }U.
\]
By definition, the $1$-form
\[
\beta=b\nu
\]
extends smoothly across $\mathcal S$. Therefore
\[
\bar g=g+\beta\otimes\beta
\]
extends smoothly across $\mathcal S$ as a Riemannian metric on $M^\circ$.
Since $\mathcal S$ has empty interior, $U$ is dense in $M^\circ$. The scalar
curvature of a smooth metric is continuous, and therefore
\[
R_{\bar g}\geq0
\]
on all of $M^\circ$.
\end{proof}

\section{ADM Energy and Horizon Completion}

Theorem~\ref{thm:MP-nonnegative-scalar} supplies the primary geometric
input for reducing the Penrose inequality for maximally parameterized
initial data to the Riemannian Penrose inequality.  It remains to
understand the behavior of the comparison metric $\bar g$ at spatial
infinity and at the \rev{marginally outer trapped boundary}.

\begin{remark}
    \rev{A sufficient asymptotic hypothesis is
\[
g-\delta=O_2(r^{-p}),\qquad \beta=O_2(r^{-q}),\qquad
p,q>\frac{n-2}{2},\qquad R_g\in L^1.
\]
Then $\bar g-g=O_2(r^{-2q})$, and the difference of the ADM fluxes is
$O(r^{n-2-2q})\to0$. Moreover,
$R_{\bar g}-R_g=O(r^{-2q-2})$, which is integrable at infinity;
the smooth completion ensures integrability on the remaining compact
region. Thus $\bar g$ has the required asymptotic regularity and
integrable scalar curvature. In particular, if $\beta$ has compact
support, or satisfies this decay, then
\begin{equation}\label{eq:sameenergy}
   E_{\rm ADM}(\bar g)=E_{\rm ADM}(g).
\end{equation}}
\end{remark}

In order to complete the proof of the Penrose inequality, we must
ensure that the deformed metric $\bar g$ extends to the boundary of the
initial data set and that the resulting boundary is minimal with
respect to $\bar g$.

Let $\Sigma_r$ be a regular leaf on which $H_{\Sigma_r}\neq0$. The
trace equation $\rev{\mathrm{I}}=0$ gives
\begin{equation}\label{eq:qrelation}
\frac{\operatorname{tr}_{\Sigma_r}k}{H_{\Sigma_r}}
=
\frac{b}{W}.
\end{equation} Recall that the null expansions are
\[
\theta_+
=
H_{\Sigma_r}+\operatorname{tr}_{\Sigma_r}k,
\qquad
\theta_-
=
H_{\Sigma_r}-\operatorname{tr}_{\Sigma_r}k.
\]
Using \eqref{eq:qrelation}, define
\begin{equation}\label{eq:uidentity}
u
\coloneq
W^{-2}
=
1-
\left(
\frac{\operatorname{tr}_{\Sigma_r}k}{H_{\Sigma_r}}
\right)^2
=
\frac{\theta_+\theta_-}{H_{\Sigma_r}^2}.
\end{equation} Suppose $\Sigma_r$ approaches a future MOTS $\Sigma_0$ with
$H_{\Sigma_0}\neq0$. Then
\[
\operatorname{tr}_{\Sigma_0}k=-H_{\Sigma_0},
\]
and hence \eqref{eq:qrelation} gives
\[
\frac{b}{\sqrt{1+b^2}}\longrightarrow -1.
\]
Since
\[
b\longmapsto \frac{b}{\sqrt{1+b^2}}
\]
takes values in $(-1,1)$ and approaches $-1$ only as $b\to-\infty$,
we obtain
\[
b\to-\infty,
\qquad
u=W^{-2}\to0.
\]
Similarly, if $\Sigma_r$ approaches a MITS $\Sigma_0$ with
$H_{\Sigma_0}\neq0$, then
\[
b\to+\infty,
\qquad
u\to0.
\]

The induced metric on each leaf $\Sigma_r$ is unchanged by the
deformation, since $\bar g|_{T\Sigma_r}=g|_{T\Sigma_r}=\gamma_r$.
Moreover, since
\[
\bar g(\nu,\nu)=1+b^2=W^2,
\]
the $\bar g$-unit normal to $\Sigma_r$ is
\[
\bar\nu=\frac{1}{W}\nu.
\]
For tangent vectors $X,Y\in T\Sigma_r$, the corresponding second
fundamental forms satisfy
\[
\bar P(X,Y)=\frac{1}{W}P(X,Y).
\]
Indeed, using the Lie-derivative characterization of the second
fundamental form,
\[
\bar P(X,Y)
=
\frac12(\mathcal L_{\bar\nu}\bar g)(X,Y)
=
\frac{1}{2W}(\mathcal L_{\nu}g)(X,Y)
=
\frac1W P(X,Y).
\]
Taking the trace with respect to the common induced metric $\gamma_r$
therefore gives
\begin{equation}\label{eq:Hbar}
\bar H_{\Sigma_r}
=
\frac{H_{\Sigma_r}}{W}
=
H_{\Sigma_r}\sqrt{u}.
\end{equation} Thus, provided $H_{\Sigma_r}$ remains bounded as $r\to0$, we have that
$u\to0$ implies
\[
\bar H_{\Sigma_r}\longrightarrow0.
\]
This suggests that the limiting horizon should become a minimal
hypersurface for the deformed metric. Equation \eqref{eq:Hbar}, however, does not determine whether the limiting leaf occurs at finite or infinite $\bar g$-distance.

Choose a smooth local defining function $r\ge0$ for the boundary on a foliated collar and write
\begin{equation}\label{eq:collar-g}
   g=\lambda(r,x)^{-2}dr^2+\gamma_r(x),
   \qquad
   \nu=\frac{dr}{\lambda(r,x)}.
\end{equation}
Since $1+b^2=W^2=u^{-1}$,
\begin{equation}\label{eq:gbar-collar}
   \bar g=\frac{dr^2}{\lambda(r,x)^2u(r,x)}+\gamma_r(x).
\end{equation}
The order of vanishing of $u$ is therefore the geometric datum that controls the inner end.

\begin{proposition}[Simple-zero completion]\label{prop:simplezero}
Assume that \eqref{eq:collar-g} holds on $[0,\varepsilon)\times\Sigma_0$, with $\lambda>0$ and $\gamma_r$ smooth up to $r=0$.  Suppose $u$ extends smoothly to the closed collar and
\begin{equation}\label{eq:simplezero}
   u(0,x)=0,
   \qquad
   \partial_r u(0,x)>0
\end{equation}
for every $x\in\Sigma_0$.  Then $(M^\circ,\bar g)$ admits a \rev{smooth finite-distance completion up to $\Sigma_0$}.  The completed boundary is totally geodesic, and its induced metric is $\gamma_0=g|_{T\Sigma_0}$.
\end{proposition}

\begin{proof}
By the Hadamard lemma, the simple-zero hypothesis gives
\[
    u(r,x)=r\,a(r,x),
\]
where $a$ is smooth and $a(0,x)>0$.  Define the variable $s=\sqrt r$ so that \eqref{eq:gbar-collar} becomes
\begin{equation}\label{eq:gbar-s}
   \bar g=\frac{4}{\lambda(s^2,x)^2a(s^2,x)}\,ds^2+\gamma_{s^2}(x).
\end{equation}After shrinking the collar if necessary, we may assume $a>0$.
We regard $s$ as the boundary defining coordinate for the completed
manifold; in particular, the smooth structure at the completed inner
end is the one determined by $(s,x)$ rather than by $(r,x)$.

In these coordinates the coefficient of $ds^2$ is smooth and strictly
positive at $s=0$, and hence the boundary lies at finite $\bar g$-distance.
Indeed, along a curve with $x$ fixed the distance to $s=0$ is bounded by
a constant multiple of
\[
\int_0^{s_0} ds<\infty.
\]
\rev{Every coefficient in \eqref{eq:gbar-s} is smooth up to $s=0$,
so this defines a smooth completion with boundary. In fact, the
coefficients are even in $s$, giving a smooth reflected extension
through $s=0$ as an additional property.} Because
\[
   \partial_s\gamma_{s^2}\big|_{s=0}=0,
\]
the second fundamental form of the hypersurface $\{s=0\}$ vanishes.  Thus the completed boundary is totally geodesic, in particular minimal.  Its tangential metric is $\gamma_0$, so its area agrees with its original $g$-area.
\end{proof}

\rev{Only this completed exterior is used in the Riemannian Penrose
argument; reflection is optional.}

\begin{remark}[\rev{Relation with MOTS stability}]\label{rem:stability}
If the leaves vary with normal speed $\rev{\psi}$ and $L$ denotes the MOTS stability operator with the convention $\frac d{dt}\theta_+|_{t=0}=L\psi$, then at a future MOTS with $H>0$ equation \eqref{eq:uidentity} gives
\begin{equation}\label{eq:u-stability}
   \left.\frac d{dt}u\right|_{t=0}=\frac{2}{H}\,L\psi.
\end{equation}
Thus the simple-zero condition means precisely that the chosen foliation moves strictly into the untrapped region to first order.  In particular, a strictly stable MOTS has a positive principal eigenfunction furnishing such a variation; see Andersson--Mars--Simon for the stability-operator framework \cite{AnderssonMarsSimon2008}.  This observation does not by itself show that an arbitrary preassigned maximally parameterized foliation is the principal-eigenfunction foliation.
\end{remark}

\rev{Let $\Sigma_0=\partial M$ denote the designated compact MOTS
boundary.} Define the minimum enclosing area by
\begin{equation}\label{eq:Amin-g}
\mathcal A_g(\Sigma_0)
:=
\inf\bigl\{
|S|_g \;\bigm|\;
\rev{S\subset M^\circ
\text{ is a smooth compact hypersurface enclosing }\Sigma_0}
\bigr\}.
\end{equation}
Define $\mathcal A_{\bar g}(\Sigma_0)$ analogously \rev{after either smooth
minimal completion allowed in Theorem~\ref{thm:main}}. Since
\[
\bar g=g+b^2\nu\otimes\nu\geq g
\]
as quadratic forms in the interior, and the completed boundary has the
same induced metric as $\Sigma_0$ in $g$, we have
\begin{equation}\label{eq:Acompare}
|S|_{\bar g}\geq |S|_g
\end{equation}
for every admissible enclosure $S$. Therefore
\begin{equation}\label{eq:Acompare-inf}
\mathcal A_{\bar g}(\Sigma_0)
\geq
\mathcal A_g(\Sigma_0).
\end{equation}
To apply the Riemannian Penrose inequality, we use the standard
minimizing-enclosure reduction.

\begin{revision}
\begin{proposition}[Enclosing-area reduction in all dimensions]\label{prop:minhull}
Let $(M^n,g)$, $n\geq3$, be asymptotically flat, complete up to its
compact smooth minimal boundary $\Sigma_0$, and satisfy $R_g\geq0$. Then
\[
E_{\rm ADM}\geq\frac12
\left(\frac{\mathcal A_g(\Sigma_0)}{\omega_{n-1}}\right)^{\frac{n-2}{n-1}}.
\]
\end{proposition}

\begin{proof}
Take the outermost perimeter-minimizing enclosure $S_*$ of $\Sigma_0$;
see \cite[Definition~2.2 and Remark~2.3]{BiZhu2026}.
One may attach an auxiliary smooth collar behind $\Sigma_0$ to formulate the
containment constraint and use a large mean-convex coordinate sphere
as an outer barrier. Compactness and lower semicontinuity of perimeter
give a minimizer; taking the outermost minimizer makes it outer minimizing.
The auxiliary extension is only used for this variational formulation;
no curvature hypothesis on it enters the argument.

Away from $\Sigma_0$, the minimizing frontier has a smooth minimal regular
part and a singular set of Hausdorff dimension at most $n-8$ \cite[Theorem~2.18]{FogagnoloMazzieri2022}.
At contact with the smooth minimal obstacle $\Sigma_0$, obstacle regularity
and the strong maximum principle give smooth contact; a component
may coincide with a component of $\Sigma_0$. Thus $S_*$ has precisely the
boundary regularity allowed in Theorem~\ref{thm:RPI}.

Smooth exterior approximation of finite-perimeter enclosures, with
any contact portions pushed a small distance outward, gives smooth
compact enclosures $S_j$ of $\Sigma_0$ such that
\[
 |S_j|_g\longrightarrow\mathcal H_g^{n-1}(S_*).
\]
Consequently $\mathcal A_g(\Sigma_0)\leq\mathcal H_g^{n-1}(S_*)$.
This approximation step is needed because the definition of
$\mathcal A_g$ uses smooth competitors even when $S_*$ is singular.
The exterior of $S_*$ containing the asymptotically flat
end has unchanged ADM energy and nonnegative scalar curvature. Applying
Theorem~\ref{thm:RPI} to that exterior yields
\[
E_{\rm ADM}\geq\frac12
\left(\frac{\mathcal H_g^{n-1}(S_*)}{\omega_{n-1}}\right)^{\frac{n-2}{n-1}}
\geq\frac12
\left(\frac{\mathcal A_g(\Sigma_0)}{\omega_{n-1}}\right)^{\frac{n-2}{n-1}}.\qedhere
\]
\end{proof}
\end{revision}

\section{Proof of Penrose Inequality for Maximally Parameterized Data}

We now have all of the pieces required to prove the main theorem: a Penrose inequality for maximally parameterized initial data. 

\begin{revision}
\begin{theorem}[Nondegenerate Penrose inequality for maximally parameterized data]
\label{thm:main}
\leavevmode\par\nobreak\noindent
Let $(M^n,g,k)$, \rev{$n\geq3$}, be an asymptotically flat
maximally parameterized initial data set satisfying the dominant energy
condition, with compact MOTS boundary $\Sigma_0=\partial M$
(or a MITS boundary after time reversal). Assume the following:
\begin{enumerate}
    \item on a foliated collar of each boundary component,
    \[
    u=(1+b^2)^{-1}
    \]
    extends smoothly to the boundary and has a simple transverse zero
    in the sense of Proposition~\ref{prop:simplezero};
    \item the comparison metric $\bar g$ is asymptotically flat and
    \begin{equation}\label{eq:masspreservehyp}
    E_{\rm ADM}(\bar g)=E_{\rm ADM}(g).
    \end{equation}
\end{enumerate}
Alternatively, item~1 may be replaced by the assumption that $\bar g$
admits a smooth completion up to $\Sigma_0$ with minimal boundary
and induced metric $g|_{T\Sigma_0}$; item~2 is retained in both
cases. Then
\begin{equation}\label{eq:energyPI}
E_{\rm ADM}(g)
\geq
\frac12
\left(
\frac{\mathcal A_g(\Sigma_0)}{\omega_{n-1}}
\right)^{\frac{n-2}{n-1}}.
\end{equation}
\end{theorem}
\end{revision}

\begin{proof}
By Theorem~\ref{thm:MP-nonnegative-scalar},
\[
R_{\bar g}\geq0
\]
on $M^\circ$. \rev{Under item~1, Proposition~\ref{prop:simplezero}
gives a smooth totally geodesic completion with induced metric
$g|_{T\Sigma_0}$. Under the alternative boundary hypothesis, the smooth
minimal boundary and equality of tangential metrics are assumed
directly. In both cases the scalar curvature remains nonnegative up
to the boundary by continuity.}

Applying Proposition~\ref{prop:minhull} to the completed metric
$\bar g$ gives
\[
E_{\rm ADM}(\bar g)
\geq
\frac12
\left(
\frac{\mathcal A_{\bar g}(\Sigma_0)}{\omega_{n-1}}
\right)^{\frac{n-2}{n-1}}.
\]
By the area comparison above,
\[
\mathcal A_{\bar g}(\Sigma_0)
\geq
\mathcal A_g(\Sigma_0),
\]
and by the asymptotic hypothesis,
\[
E_{\rm ADM}(\bar g)=E_{\rm ADM}(g).
\]
\rev{The desired inequality \eqref{eq:energyPI} now follows.}
\end{proof}

\section{Rigidity and Further Questions}\label{sec:open}

Our construction of maximally parameterized presentations and \rev{initial} data isolates several problems that should not be conflated with the scalar curvature calculation.

\subsection{Degenerate horizons and cylindrical ends}\label{sec:degenerate}

\rev{The simple-zero branch of Theorem~\ref{thm:main} treats the nondegenerate
regime; its regular minimal boundary alternative does not address
degenerate horizon blow-up.}  The remaining horizon problem is to classify the metric \eqref{eq:gbar-collar} when $du=0$ at the MOTS.  A quadratic zero is the natural cylindrical model, but a proof that an arbitrary degenerate maximally parameterized horizon has the quantitative asymptotics required by a cylindrical-end Penrose theorem is still missing.  In dimension three, the theorem of Jaracz \cite{Jaracz2020} provides an appropriate target once those asymptotics and the generalized minimizing condition are verified.  In higher dimensions an analogous cylindrical-end statement would need to be supplied or derived from an all-dimensional Riemannian Penrose theorem by a truncation/doubling argument.

Equation \eqref{eq:u-stability} suggests that this dichotomy may be related to MOTS stability.  If the actual foliation lapse $\psi$ satisfies $L\psi>0$, then the horizon is in the simple-zero regime.  The marginal case $L\psi=0$ is the natural place to expect cylindrical behavior.  Turning this observation into a complete boundary classification is a concrete next problem.

\subsection{Global surjectivity of the parametrization}\label{sec:surjectivity}

The count of local functional degrees of freedom does not show that a prescribed global $k$ can always be represented in the form \eqref{eq:kMP}.  The open questions include the existence of a suitable integrable one-form $\nu$, the behavior of its singular set, and the solvability of the weighted equation for $Q$ across leaves of changing topology.  Thus `maximally parameterized' should be read as a local functional statement, not as a theorem that the construction parametrizes all initial data.

\subsection{Nonzero ADM linear momentum}

\rev{The Riemannian comparison directly controls $E_{\rm ADM}$.
An ADM mass conclusion for nonzero ADM linear momentum requires an admissible
rest-frame boost and any necessary approximation that retain the
comparison hypotheses and control minimum enclosing area. Establishing
these properties within the maximally parameterized class remains an
additional problem.}

\subsection{Equality and rigidity}

Suppose equality holds in Theorem~\ref{thm:main}. Then equality must
hold at each step of the comparison argument. In particular
\[
\mathcal A_{\bar g}(\Sigma_0)
=
\mathcal A_g(\Sigma_0),
\]
and equality holds in the Riemannian Penrose inequality for the exterior
comparison geometry. \rev{The rigidity clause of
\cite[Theorem~1.2]{BiZhu2026} identifies this exterior with a Riemannian
Schwarzschild exterior in every dimension.}
\rev{In particular}, on this exterior region we therefore have
\[
R_{\bar g}=0.
\]
The \rev{gauge-invariant Schoen--Yau identity} gives
\[
0
=
16\pi\bigl(\mu-J(v)\bigr)
+
|\Delta|_g^2.
\]
Since both terms on the right-hand side are nonnegative under the
dominant energy condition, it follows that
\begin{equation}\label{eq:equalitydelta}
\Delta=0,
\qquad
\mu-J(v)=0,
\end{equation}
\rev{on the regular presentation set in the exterior comparison region}.
\rev{At every such point, where $b$ is finite,}
\[
|v|_g<1.
\]
Together with the dominant energy condition,
\eqref{eq:equalitydelta} therefore implies
\[
\mu=0,
\qquad
J=0.
\]
\rev{The original constraint densities are smooth, so their vanishing
extends across the singular presentation set by continuity. Thus
equality forces vacuum data throughout the comparison exterior and
vanishing discrepancy wherever its intrinsic presentation is defined.}

What remains is a genuine spacetime rigidity problem: one must
understand solutions of
\[
\Delta=0
\]
over the Schwarzschild comparison geometry, and determine when the
original initial data $(M,g,k)$ arise as a spacelike slice of the
Schwarzschild spacetime. We do not claim this final reconstruction
here.

\appendix

\section{\rev{Derivation of the New Schoen--Yau Identity and Its Gauge-Invariant Form}}
\label{app:one-parameter-identity}

In this appendix, we give the calculations underlying
Theorem~\ref{thm:GSY-t-extended} and the \rev{gauge-invariant Schoen--Yau identity}
\eqref{eq:main-identity}.
Let $f,\phi\colon M\to\mathbb R$ be smooth with $\phi>0$, and set
\[
\bar g=g+\phi^2\,df\otimes df,
\qquad
W=\sqrt{1+\phi^2|df|_g^2},
\qquad
v=\frac{\phi\nabla f}{W}.
\]
On the set where $df\neq0$, define
\[
\nu=\frac{df}{|df|_g},
\]
so that
\[
v=\frac{\phi|df|_g}{W}\,\nu.
\]
We write
\[
k_{\nu\nu}=k(\nu,\nu),
\qquad
h_{\nu\nu}=h(\nu,\nu),
\]
and define
\[
\sigma=k_{\nu\nu}-h_{\nu\nu}.
\]
Recall that
\[
\delta(t)
=
h-k+t\sigma\,\nu\otimes\nu.
\]
It will also be convenient to define
\begin{equation}\label{eq:H-t-def}
H(t)
\coloneq
h+t\sigma\,\nu\otimes\nu,
\end{equation}
so that
\[
\delta(t)=H(t)-k.
\]
Since
\[
\bar g^{ij}
=
g^{ij}-v^iv^j,
\]
we have
\begin{equation}\label{eq:bar-nu}
\bar g^{ij}\nu_j
=
\frac{\nu^i}{W^2},
\qquad
|\nu|_{\bar g}^2
=
\frac{1}{W^2}.
\end{equation}
Moreover,
\begin{equation}\label{eq:delta-nunu}
\delta(t)_{\nu\nu}
=
h_{\nu\nu}-k_{\nu\nu}+t\sigma
=
(t-1)\sigma.
\end{equation}

\subsection{Variation identities}

\medskip

\noindent\textbf{Identity 1.}
\begin{equation}\label{eq:identity-1}
W^4\frac{d}{dt}|\delta(t)|_{\bar g}^2
=
2\sigma^2(t-1)
\end{equation}

\begin{proof}
Since $\bar g$ is independent of $t$ and
\[
\frac{d}{dt}\delta(t)
=
\sigma\,\nu\otimes\nu,
\]
we compute
\begin{align*}
W^4\frac{d}{dt}|\delta(t)|_{\bar g}^2
&=
2W^4\bar g^{im}\bar g^{jn}
\delta(t)_{ij}
\frac{d}{dt}\delta(t)_{mn}\\
&=
2\sigma W^4
\bar g^{im}\bar g^{jn}
\delta(t)_{ij}\nu_m\nu_n\\
&=
2\sigma\delta(t)_{\nu\nu}\\
&=
2\sigma^2(t-1),
\end{align*}
where we used \eqref{eq:bar-nu} in the third line and
\eqref{eq:delta-nunu} in the final line.
\end{proof}

\medskip

\noindent\textbf{Identity 2.}
\begin{equation}\label{eq:identity-2}
W^4\frac{d}{dt}
\left(
2|\delta(t)(v,\cdot)|_{\bar g}^2
\right)
=
4\phi^2|df|_g^2\sigma^2(t-1)
\end{equation}

\begin{proof}
Since $v$ and $\bar g$ are independent of $t$,
\begin{align*}
W^4\frac{d}{dt}
\left(
2|\delta(t)(v,\cdot)|_{\bar g}^2
\right)
&=
4W^4
\bar g^{ij}v^kv^\ell
\delta(t)_{ik}
\frac{d}{dt}\delta(t)_{j\ell}\\
&=
4\sigma W^4
\bar g^{ij}v^kv^\ell
\delta(t)_{ik}\nu_j\nu_\ell.
\end{align*}
Using
\[
v^i
=
\frac{\phi|df|_g}{W}\nu^i
\]
and \eqref{eq:bar-nu}, this becomes
\begin{align*}
W^4\frac{d}{dt}
\left(
2|\delta(t)(v,\cdot)|_{\bar g}^2
\right)
&=
4\phi^2|df|_g^2
\sigma\,\delta(t)_{\nu\nu}\\
&=
4\phi^2|df|_g^2
\sigma^2(t-1),
\end{align*}
as claimed.
\end{proof}

\medskip

\noindent\textbf{Identity 3.}
\begin{equation}\label{eq:identity-3}
W^4\frac{d}{dt}
\left(
\operatorname{tr}_{\bar g}H(t)
\right)^2
=
2\sigma W^2\operatorname{tr}_{\bar g}h
+
2t\sigma^2
\end{equation}

\begin{proof}
By \eqref{eq:H-t-def},
\[
\frac{d}{dt}
\operatorname{tr}_{\bar g}H(t)
=
\sigma\bar g^{ij}\nu_i\nu_j
=
\frac{\sigma}{W^2}.
\]
Moreover,
\[
\operatorname{tr}_{\bar g}H(t)
=
\operatorname{tr}_{\bar g}h
+
\frac{t\sigma}{W^2}.
\]
Therefore
\begin{align*}
W^4\frac{d}{dt}
\left(
\operatorname{tr}_{\bar g}H(t)
\right)^2
&=
2W^4
\operatorname{tr}_{\bar g}H(t)
\frac{\sigma}{W^2}\\
&=
2\sigma W^2
\left(
\operatorname{tr}_{\bar g}h
+
\frac{t\sigma}{W^2}
\right)\\
&=
2\sigma W^2\operatorname{tr}_{\bar g}h
+
2t\sigma^2.\qedhere
\end{align*}
\end{proof}

\medskip

\noindent\textbf{Identity 4.}
\begin{equation}\label{eq:identity-4}
W^4\frac{d}{dt}
\left[
2k(v,v)\operatorname{tr}_{\bar g}\delta(t)
\right]
=
2\phi^2|df|_g^2\sigma k_{\nu\nu}
\end{equation}

\begin{proof}
Since $k(v,v)$ is independent of $t$,
\begin{align*}
W^4\frac{d}{dt}
\left[
2k(v,v)\operatorname{tr}_{\bar g}\delta(t)
\right]
&=
2W^4 k(v,v)
\frac{d}{dt}
\operatorname{tr}_{\bar g}\delta(t)\\
&=
2W^4 k(v,v)
\frac{\sigma}{W^2}\\
&=
2\sigma W^2 k(v,v).
\end{align*}
Since
\[
v
=
\frac{\phi|df|_g}{W}\nu,
\]
we have
\[
k(v,v)
=
\frac{\phi^2|df|_g^2}{W^2}k_{\nu\nu}.
\]
Substitution gives \eqref{eq:identity-4}.
\end{proof}

\medskip

\noindent\textbf{Identity 5.}
\begin{equation}\label{eq:identity-5}
W^4\frac{d}{dt}
\left[
\frac{2\sigma^2t(1-t)}{W^2}
\right]
=
2\sigma^2W^2(1-2t)
\end{equation}

\begin{proof}
The quantities $W$ and $\sigma$ are independent of $t$, and hence
\[
W^4\frac{d}{dt}
\left[
\frac{2\sigma^2t(1-t)}{W^2}
\right]
=
2\sigma^2W^2(1-2t).\qedhere
\]
\end{proof}

\medskip

\noindent\textbf{Identity 6.}
\begin{equation}\label{eq:identity-6}
W^2\operatorname{div}_{\bar g}(\phi v)
=
\operatorname{div}_g(\phi v)
-
\frac{\phi\,v(\phi^2|df|_g^2)}{2W^2}
\end{equation}

\begin{proof}
Using
\[
\sqrt{\det\bar g}
=
W\sqrt{\det g}
\]
and
\[
\bar g^{ij}v_j
=
\frac{v^i}{W^2},
\]
we compute
\begin{align*}
W^2\operatorname{div}_{\bar g}(\phi v)
&=
\frac{W^2}{\sqrt{\det\bar g}}
\partial_i
\left(
\sqrt{\det\bar g}\,
\bar g^{ij}\phi v_j
\right)\\
&=
\frac{W}{\sqrt{\det g}}
\partial_i
\left(
\sqrt{\det g}\,
\frac{\phi v^i}{W}
\right)\\
&=
\operatorname{div}_g(\phi v)
+
W\phi v^i\partial_i\left(\frac1W\right)\\
&=
\operatorname{div}_g(\phi v)
-
\frac{\phi\,v(W)}{W}.
\end{align*}
Since
\[
W^2=1+\phi^2|df|_g^2,
\]
we have
\[
v(W)
=
\frac{v(\phi^2|df|_g^2)}{2W},
\]
which gives \eqref{eq:identity-6}.
\end{proof}

\medskip

\noindent\textbf{Identity 7.}
\begin{equation}\label{eq:identity-7}
\operatorname{div}_g(\phi v)
=
\phi\operatorname{tr}_g h
-
\frac{\phi\,v(\phi^2|df|_g^2)}{2W^2}
\end{equation}

\begin{proof}
Since
\[
\phi v
=
\frac{\phi^2\nabla f}{W},
\]
we have
\begin{align*}
\operatorname{div}_g(\phi v)
&=
\frac{\phi^2}{W}\Delta_g f
+
\frac{2\phi}{W}
\langle d\phi,df\rangle_g
+
\phi^2
\left\langle
df,d\left(\frac1W\right)
\right\rangle_g.
\end{align*}
By the definition of $h$,
\[
\phi\operatorname{tr}_g h
=
\frac{\phi^2}{W}\Delta_g f
+
\frac{2\phi}{W}
\langle d\phi,df\rangle_g.
\]
Therefore
\begin{align*}
\operatorname{div}_g(\phi v)
&=
\phi\operatorname{tr}_g h
+
\phi^2
\left\langle
df,d\left(\frac1W\right)
\right\rangle_g\\
&=
\phi\operatorname{tr}_g h
-
\frac{\phi^2}{2W^3}
\left\langle
df,d(\phi^2|df|_g^2)
\right\rangle_g.
\end{align*}
Finally, using
\[
v
=
\frac{\phi\nabla f}{W},
\]
the last term becomes
\[
-\frac{\phi\,v(\phi^2|df|_g^2)}{2W^2},
\]
which proves the identity.
\end{proof}

\medskip

\noindent\textbf{Identity 8.}
\begin{equation}\label{eq:identity-8}
\begin{aligned}
W^4\frac{d}{dt}
\bigg[
-\frac{2}{\phi}
\operatorname{div}_{\bar g}
\bigl(
\phi\delta(t)(v,\cdot)
\bigr)
+
2v\bigl(
\operatorname{tr}_{\bar g}\delta(t)
\bigr)
\bigg]
=
-2\sigma W^2\operatorname{tr}_g h
\end{aligned}
\end{equation}

\begin{proof}
We compute the two terms separately. Since
\[
\frac{d}{dt}\delta(t)
=
\sigma\nu\otimes\nu
\]
and
\[
\nu(v)
=
\frac{\phi|df|_g}{W},
\]
we have
\[
\frac{d}{dt}\bigl(\delta(t)(v,\cdot)\bigr)
=
\sigma v.
\]
Therefore
\begin{align*}
&W^4\frac{d}{dt}
\left[
-\frac{2}{\phi}
\operatorname{div}_{\bar g}
\bigl(
\phi\delta(t)(v,\cdot)
\bigr)
\right]\\
&\qquad
=
-\frac{2W^4}{\phi}
\operatorname{div}_{\bar g}
(\phi\sigma v)\\
&\qquad
=
-\frac{2\sigma W^4}{\phi}
\operatorname{div}_{\bar g}(\phi v)
-
2W^4
\langle d\sigma,v\rangle_{\bar g}.
\end{align*}
Since
\[
\bar g^{ij}v_j
=
\frac{v^i}{W^2},
\]
we have
\[
\langle d\sigma,v\rangle_{\bar g}
=
\frac{v(\sigma)}{W^2}.
\]
Hence
\begin{equation}\label{eq:identity-8-first}
W^4\frac{d}{dt}
\left[
-\frac{2}{\phi}
\operatorname{div}_{\bar g}
\bigl(
\phi\delta(t)(v,\cdot)
\bigr)
\right]
=
-\frac{2\sigma W^4}{\phi}
\operatorname{div}_{\bar g}(\phi v)
-
2W^2v(\sigma).
\end{equation}

For the second term,
\begin{align*}
W^4\frac{d}{dt}
\left[
2v\bigl(
\operatorname{tr}_{\bar g}\delta(t)
\bigr)
\right]
&=
2W^4v\left(\frac{\sigma}{W^2}\right)\\
&=
2W^2v(\sigma)
+
2\sigma W^4v(W^{-2})\\
&=
2W^2v(\sigma)
-
2\sigma\,v(\phi^2|df|_g^2).
\end{align*}
Adding this to \eqref{eq:identity-8-first} gives
\begin{align*}
& W^4\frac{d}{dt}
\bigg[
-\frac{2}{\phi}
\operatorname{div}_{\bar g}
\bigl(
\phi\delta(t)(v,\cdot)
\bigr)
+
2v\bigl(
\operatorname{tr}_{\bar g}\delta(t)
\bigr)
\bigg]\\
&\qquad
=
-\frac{2\sigma W^4}{\phi}
\operatorname{div}_{\bar g}(\phi v)
-
2\sigma\,v(\phi^2|df|_g^2).
\end{align*}
Using Identity~\eqref{eq:identity-6},
\begin{align*}
&=
-\frac{2\sigma W^2}{\phi}
\left[
\operatorname{div}_g(\phi v)
-
\frac{\phi\,v(\phi^2|df|_g^2)}{2W^2}
\right]
-
2\sigma\,v(\phi^2|df|_g^2)\\
&=
-\frac{2\sigma W^2}{\phi}
\operatorname{div}_g(\phi v)
-
\sigma\,v(\phi^2|df|_g^2).
\end{align*}
Finally, Identity~\eqref{eq:identity-7} gives
\begin{align*}
&=
-2\sigma W^2\operatorname{tr}_g h
+
\sigma\,v(\phi^2|df|_g^2)
-
\sigma\,v(\phi^2|df|_g^2)\\
&=
-2\sigma W^2\operatorname{tr}_g h.
\end{align*}
This proves \eqref{eq:identity-8}.
\end{proof}

\subsection{\rev{Proof of the New Schoen--Yau Identity (General $t$)}}

\begin{proof}[Proof of Theorem~\ref{thm:GSY-t-extended}]
Let $P(t)$ denote the right-hand side of
\eqref{eq:GSY-param-t}.  At $t=0$, we have
\[
\delta(0)=h-k,
\]
and \eqref{eq:GSY-param-t} is precisely the generalized
Schoen--Yau identity of Bray--Khuri.  It therefore suffices to show
that
\[
P'(t)=0.
\]

Since
\[
\delta(t)=H(t)-k,
\]
we may rewrite
\begin{align*}
&
\bigl(
\operatorname{tr}_{\bar g}\delta(t)
\bigr)^2
+
2\bigl(
\operatorname{tr}_{\bar g}k+k(v,v)
\bigr)
\operatorname{tr}_{\bar g}\delta(t)
\\
&\qquad
=
\bigl(
\operatorname{tr}_{\bar g}H(t)
\bigr)^2
-
\bigl(
\operatorname{tr}_{\bar g}k
\bigr)^2
+
2k(v,v)
\operatorname{tr}_{\bar g}\delta(t).
\end{align*}
The quantities $\mu$, $J$, $v$, $\bar g$, $k$, $\sigma$, and $W$
are all independent of $t$.  Using
Identities~\eqref{eq:identity-1}--\eqref{eq:identity-5} and
\eqref{eq:identity-8}, we obtain
\begin{align*}
W^4P'(t)
&=
2\sigma^2(t-1)
+
4\phi^2|df|_g^2\sigma^2(t-1)
-
2\sigma W^2\operatorname{tr}_g h\\
&\quad
+
2\sigma^2W^2(1-2t)
+
2\sigma W^2\operatorname{tr}_{\bar g}h
+
2t\sigma^2\\
&\quad
+
2\phi^2|df|_g^2\sigma k_{\nu\nu}.
\end{align*}
Using
\[
W^2=1+\phi^2|df|_g^2
\]
and collecting the terms involving $\sigma^2$ gives
\begin{align*}
W^4P'(t)
&=
-2\phi^2|df|_g^2\sigma^2
-
2\sigma W^2\operatorname{tr}_g h\\
&\quad
+
2\sigma W^2\operatorname{tr}_{\bar g}h
+
2\phi^2|df|_g^2\sigma k_{\nu\nu}.
\end{align*}
Since
\[
\bar g^{ij}=g^{ij}-v^iv^j,
\]
we have
\[
\operatorname{tr}_{\bar g}h
=
\operatorname{tr}_g h-h(v,v).
\]
Thus
\begin{align*}
W^4P'(t)
&=
-2\phi^2|df|_g^2\sigma^2
-
2\sigma W^2h(v,v)
+
2\phi^2|df|_g^2\sigma k_{\nu\nu}.
\end{align*}
But
\[
h(v,v)
=
\frac{\phi^2|df|_g^2}{W^2}h_{\nu\nu},
\]
and hence
\begin{align*}
W^4P'(t)
&=
-2\phi^2|df|_g^2\sigma^2
+
2\phi^2|df|_g^2
\sigma
\bigl(
k_{\nu\nu}-h_{\nu\nu}
\bigr)\\
&=
-2\phi^2|df|_g^2\sigma^2
+
2\phi^2|df|_g^2\sigma^2\\
&=
0.
\end{align*}
Therefore $P(t)$ is independent of $t$, and since
\eqref{eq:GSY-param-t} holds at $t=0$, it holds for every
$t\in\mathbb R$.
\end{proof}

\subsection{\rev{The Gauge-Invariant Schoen--Yau Identity ($t=1$)}}
\label{app:gauge-invariant-identity}

\begin{revision}
We derive \eqref{eq:main-identity} locally, first where $b>0$.
Write $\nu=dr/|dr|_g$, set $f=r$ and $\phi=b/|dr|_g>0$, and obtain
\[
\beta=b\nu=\phi\,df,\qquad
\nu=\frac{df}{|df|_g},\qquad
b=\phi|df|_g,\qquad
W=\sqrt{1+b^2},\qquad
v=\frac{b}{W}\nu.
\]
The region where $b<0$ is handled by simultaneously reversing $(\nu,b)$: this leaves $\beta$, $v$, $\Delta$, $\rev{\mathrm{I}}$, and $\rev{\mathrm{II}}$ unchanged, while $P$ and $H_\Sigma$ reverse sign and $\kappa$ is unchanged. Thus the positive-$b$ calculation covers both signs. Zeros are treated at the end using the smooth intrinsic definition \eqref{eq:Delta-intrinsic}.
On this regular graph region put $\Delta=\delta(1)$. The component calculations below identify this with \eqref{eq:Delta-intrinsic}. By the definition of $\sigma$,
\end{revision}
\begin{equation}\label{eq:Delta-nunu-zero}
\Delta(\nu,\nu)=0.
\end{equation}

Let $X,Y$ be tangent to the leaves of $\ker\nu$.  Since
\[
df=\frac{b}{\phi}\nu,
\]
the second fundamental form $h$ of the graph satisfies
\begin{equation}\label{eq:h-tangential}
h(X,Y)
=
\frac{b}{W}P(X,Y).
\end{equation}
Indeed, the terms involving $d\phi\otimes df$ vanish on tangent
vectors, while
\[
\nabla^2f(X,Y)
=
|df|_g P(X,Y).
\]

For the mixed component,
\begin{equation}\label{eq:h-mixed}
h(\nu,X)
=
\frac{d_\Sigma b(X)}{W}.
\end{equation}
To see this, write $q=|df|_g$, so that $b=\phi q$.  Since
$X(f)=0$,
\[
\nabla^2f(\nu,X)=X(q),
\]
and therefore
\[
Wh(\nu,X)
=
\phi X(q)+qX(\phi)
=
X(\phi q)
=
X(b).
\]

It follows from \eqref{eq:Delta-nunu-zero} that
\[
\operatorname{tr}_{\bar g}\Delta
=
\operatorname{tr}_g\Delta
=
\operatorname{tr}_\Sigma\Delta.
\]
Using \eqref{eq:h-tangential},
\begin{align}
\operatorname{tr}_{\bar g}\Delta
&=
\frac{b}{W}H_\Sigma
-
\operatorname{tr}_\Sigma k
\nonumber\\
&=
\rev{\mathrm{I}}.
\label{eq:Delta-trace-I}
\end{align}
Similarly, \eqref{eq:h-mixed} gives
\begin{equation}\label{eq:Delta-mixed-Q}
\Delta(\nu,\cdot)|_\Sigma
=
\frac{d_\Sigma b}{W}
-
k(\nu,\cdot)|_\Sigma.
\end{equation}

We next compare the quadratic terms in the \rev{new Schoen--Yau identity}.
Because $\Delta_{\nu\nu}=0$, decomposing $\Delta$ into tangential
and mixed components gives
\begin{equation}\label{eq:Delta-norm-conversion}
|\Delta|_{\bar g}^2
+
2|\Delta(v,\cdot)|_{\bar g}^2
=
|\Delta|_g^2.
\end{equation}
Indeed, if
\[
Q
=
\Delta(\nu,\cdot)|_\Sigma,
\]
then
\[
|\Delta|_{\bar g}^2
=
|\Delta|_{T\Sigma}^2
+
\frac{2}{W^2}|Q|_g^2,
\]
while
\[
2|\Delta(v,\cdot)|_{\bar g}^2
=
\frac{2b^2}{W^2}|Q|_g^2.
\]
Since $W^2=1+b^2$, their sum is exactly $|\Delta|_g^2$.

We now rewrite the divergence term.  Set
\[
Q
=
\Delta(\nu,\cdot)|_\Sigma.
\]
Then
\[
\Delta(v,\cdot)
=
\frac{b}{W}Q.
\]
On the region where $b\neq0$, the relation
\[
b=\phi|df|_g
\]
and Lemma~\ref{lem:kappa-log-lambda}, applied with $r=f$, give
\begin{equation}\label{eq:phi-b-kappa}
d_\Sigma\log\phi
=
d_\Sigma\log|b|-\kappa.
\end{equation}

For a tangential vector field $X$, we have
\[
\operatorname{div}_gX
=
\operatorname{div}_\Sigma X-g(X,\kappa),
\]
and, because
\[
\sqrt{\det\bar g}
=
W\sqrt{\det g},
\]
we also have
\[
\operatorname{div}_{\bar g}X
=
\operatorname{div}_gX+X(\log W).
\]
Taking
\[
X=\frac{b}{W}Q,
\]
we obtain
\begin{align*}
\frac{1}{\phi}
\operatorname{div}_{\bar g}
\left(
\phi\frac{b}{W}Q
\right)
&=
\operatorname{div}_\Sigma
\left(
\frac{b}{W}Q
\right)
-
\frac{b}{W}Q(\kappa)\\
&\quad
+
\frac{b}{W}
\langle d_\Sigma\log W,Q\rangle_g
+
\frac{b}{W}
\langle d_\Sigma\log\phi,Q\rangle_g.
\end{align*}
Expanding the first term and using
\eqref{eq:phi-b-kappa}, we find
\begin{align*}
\frac{1}{\phi}
\operatorname{div}_{\bar g}
\left(
\phi\Delta(v,\cdot)
\right)
&=
\frac{1}{W}
\bigg[
2\langle d_\Sigma b,Q\rangle_g\\
&\qquad\qquad
+
b\left(
\operatorname{div}_\Sigma Q
-
2Q(\kappa)
\right)
\bigg]\\
&=
\frac{\rev{\mathrm{II}}}{W}.
\end{align*}
\rev{The final intrinsic expression $\rev{\mathrm{II}}/W$ is smooth for a specified smooth pair $(\nu,b)$. The graph calculation applies where $b\neq0$; the full intrinsic identity at zeros is justified below.}

Finally,
\begin{equation}\label{eq:trace-k-conversion}
\operatorname{tr}_{\bar g}k+k(v,v)
=
\operatorname{tr}_g k,
\end{equation}
since
\[
\bar g^{ij}=g^{ij}-v^iv^j,
\]
and
\begin{equation}\label{eq:v-I}
v(\rev{\mathrm{I}})
=
\frac{b}{W}\nu(\rev{\mathrm{I}}).
\end{equation}

Specializing Theorem~\ref{thm:GSY-t-extended} to $t=1$ and using
\eqref{eq:Delta-trace-I},
\eqref{eq:Delta-norm-conversion},
\eqref{eq:trace-k-conversion}, and
\eqref{eq:v-I}, together with the divergence calculation above,
gives
\[
R_{\bar g}
=
16\pi
\left(
\mu-J\left(\frac{b}{W}\nu\right)
\right)
+
|\Delta|_g^2
-
\frac{2}{W}\rev{\mathrm{II}}
+
\rev{\mathrm{I}}^2
+
2\operatorname{tr}_g k\,\rev{\mathrm{I}}
+
\frac{2b}{W}\nu(\rev{\mathrm{I}}).
\]
This is precisely the \rev{gauge-invariant Schoen--Yau identity}
\eqref{eq:main-identity} \rev{on $\{b\neq0\}$. To include any point $p$ with $b(p)=0$, replace $b$ locally by $b+\varepsilon$ and evaluate the established identity at $p$, where $b+\varepsilon\neq0$. Letting $\varepsilon\to0$ proves the result at $p$, because all terms defined from $(\nu,b+\varepsilon)$ converge smoothly. This argument also covers open regions on which $b=0$, without requiring a graph representation there.}

\section{\rev{Examples of Maximally Parameterized Initial Data Sets}} 
\label{app:examples}

In this appendix, we give several classes of maximally parameterized
initial data sets.  The examples range from general graph slices of
static spacetimes to nonspherically symmetric black hole slices,
\rev{and rotating data.}  

\subsection{Graph slices of static spacetimes}

We begin with a general source of maximally parameterized initial data.
Let
\[
(\widetilde M^{n+1},\widetilde g)
=
(\mathbb R\times N^n,-\phi^2dt^2+\bar g)
\]
be a static spacetime, where $\phi>0$ and $\bar g$ are independent of
$t$.  Let
\[
M_f=\{t=f(x):x\in N\}
\]
be a smooth spacelike graph.  Its induced metric is
\begin{equation}\label{eq:static-graph-induced}
g
=
\bar g-\phi^2df\otimes df.
\end{equation}
Thus, if
\[
\beta=\phi\,df,
\]
then
\begin{equation}\label{eq:static-graph-comparison}
\bar g
=
g+\beta\otimes\beta.
\end{equation}

Let $k$ denote the second fundamental form of the graph, with the
normal chosen consistently with the graph convention used in
\rev{Section~\ref{sec:new-identities}}. Then $k$ agrees with the tensor $h$ appearing in the
generalized Jang construction.

\begin{proposition}[Static graph slices]\label{prop:static-graphs}
On every open set $U\subset M_f$ on which $df\neq0$, the induced
initial data $(U,g,k)$ admit a maximally parameterized presentation.
Moreover, the associated Riemannian comparison metric is precisely the
static spatial metric $\bar g$.
\end{proposition}

\begin{proof}
On $U$, define
\[
\nu=\frac{df}{|df|_g},
\qquad
b=\phi|df|_g.
\]
Then
\[
\beta=b\nu=\phi\,df
\]
and
\[
\bar g=g+b^2\nu\otimes\nu.
\]
Since $k=h$, we have
\[
\delta(0)=h-k=0.
\]
Furthermore,
\[
\sigma
=
k(\nu,\nu)
-
h(\nu,\nu)
=
0,
\]
and hence
\[
\Delta=\delta(1)=0.
\]

The identities in Appendix~\ref{app:gauge-invariant-identity} therefore
give
\[
\operatorname{tr}_\Sigma k
=
\frac{b}{W}H_\Sigma
\]
and
\[
k(\nu,\cdot)|_\Sigma
=
\frac{d_\Sigma b}{W}.
\]
Set
\[
A=k(\nu,\nu),
\qquad
Q=0,
\]
and let
\[
\rev{\hat k}
=
k|_{T\Sigma}
-
\frac{\operatorname{tr}_\Sigma k}{n-1}\gamma
\]
be the tangential trace-free part of $k$.  Substituting these quantities
into \eqref{eq:kMP} recovers $k$ exactly.  Thus
$(\nu,b,A,0,\rev{\hat k})$ is a maximally parameterized presentation.
\end{proof}

This construction does not impose any symmetry on $f$.  For example,
in Minkowski spacetime
\[
\widetilde g=-dt^2+\delta,
\]
an arbitrary spacelike graph $t=f(x)$ satisfies
\[
g=\delta-df\otimes df,
\qquad
\bar g=\delta.
\]
Thus arbitrary nonsymmetric spacelike graphs in Minkowski spacetime
give maximally parameterized data whose comparison metric is Euclidean.

If the critical set of $f$ has empty interior, it may be included in
the singular set $\mathcal S$ of \rev{Definition~\ref{def:MP-initial-data}}.  In examples in which $f$ is
constant on an open set, the presentation may instead be extended
across that region whenever the unit foliation form $\nu$ admits a
smooth continuation there.  The examples below will be chosen so that
this continuation is explicit.

\subsection{Nonspherically symmetric slices of Schwarzschild spacetime}
\label{sec:schwarzschild-examples}

We next consider a class which is particularly natural for the Penrose
inequality.  Let
\[
F(r)=1-\frac{2m}{r^{n-2}},
\qquad
r_+=(2m)^{1/(n-2)},
\]
and write the $(n+1)$-dimensional Schwarzschild metric in the exterior
region as
\begin{equation}\label{eq:schwarzschild-static}
\widetilde g
=
-F(r)\,dt^2
+
\frac{dr^2}{F(r)}
+
r^2\gamma_{\mathbb S^{n-1}}.
\end{equation}
The associated Riemannian Schwarzschild metric is
\begin{equation}\label{eq:riemannian-schwarzschild-example}
g_{\mathrm{Sch}}
=
\frac{dr^2}{F(r)}
+
r^2\gamma_{\mathbb S^{n-1}}.
\end{equation}

To describe slices which cross the future horizon smoothly, introduce
the ingoing Eddington--Finkelstein coordinate
\[
v=t+r^*,
\qquad
\frac{dr^*}{dr}=\frac1F.
\]
Then
\begin{equation}\label{eq:schwarzschild-EF}
\widetilde g
=
-F\,dv^2
+
2\,dv\,dr
+
r^2\gamma_{\mathbb S^{n-1}},
\end{equation}
which extends smoothly through $r=r_+$.

Let
\[
v=V(r,\omega),
\qquad
\omega\in\mathbb S^{n-1},
\]
where \rev{$V$ is a smooth function chosen so that the slice is
spacelike}. No spherical symmetry is
assumed.  Pulling back \eqref{eq:schwarzschild-EF} gives
\begin{equation}\label{eq:general-kruskal-slice}
g
=
-F\,dV\otimes dV
+
dV\otimes dr
+
dr\otimes dV
+
r^2\gamma_{\mathbb S^{n-1}}.
\end{equation}
Writing
\[
dV
=
V_r\,dr+d_\omega V,
\]
we obtain
\begin{align}
g
={}&
\left(2V_r-FV_r^2\right)dr^2
+
(1-FV_r)
\left(
dr\otimes d_\omega V
+
d_\omega V\otimes dr
\right)
\nonumber\\
&+
r^2\gamma_{\mathbb S^{n-1}}
-
F\,d_\omega V\otimes d_\omega V.
\label{eq:general-kruskal-expanded}
\end{align}
In particular, for a generic choice of $V(r,\omega)$ the induced
metric is not spherically symmetric.

On the exterior $r>r_+$, the same slice may be written as a graph
\[
t=f(r,\omega),
\qquad
f=V-r^*.
\]
Since the static lapse is $\phi=\sqrt F$, Proposition~\ref{prop:static-graphs}
applies with
\[
\beta=\sqrt F\,df.
\]
Therefore
\begin{equation}\label{eq:schwarzschild-graph-comparison}
\bar g
=
g+\beta\otimes\beta
=
g_{\mathrm{Sch}}.
\end{equation}
Thus every such slice has the same
Riemannian Schwarzschild comparison geometry.

We now compute the function
\[
u=(1+b^2)^{-1}
\]
which governs the behavior of the comparison metric at the horizon.
For a static graph,
\[
u
=
1-\phi^2|df|_{\bar g}^2.
\]
Using
\[
df
=
\left(V_r-\frac1F\right)dr+d_\omega V
\]
and \eqref{eq:riemannian-schwarzschild-example}, we obtain
\begin{align*}
1-u
&=
F|df|_{g_{\mathrm{Sch}}}^2\\
&=
\left(FV_r-1\right)^2
+
\frac{F}{r^2}|d_\omega V|_{\gamma_{\mathbb S^{n-1}}}^2.
\end{align*}
Hence
\begin{equation}\label{eq:u-general-schwarzschild}
u
=
F
\left(
2V_r
-
FV_r^2
-
\frac1{r^2}|d_\omega V|_{\gamma_{\mathbb S^{n-1}}}^2
\right).
\end{equation}
It is convenient to write
\begin{equation}\label{eq:B-schwarzschild}
B(r,\omega)
=
2V_r
-
FV_r^2
-
\frac1{r^2}|d_\omega V|_{\gamma_{\mathbb S^{n-1}}}^2,
\end{equation}
so that
\[
u=FB.
\]

The same quantity has a direct spacetime interpretation at the horizon.
Restricting \eqref{eq:general-kruskal-expanded} to $r=r_+$ gives the
matrix
\[
\begin{pmatrix}
2V_r & (d_\omega V)_B\\
(d_\omega V)_A & r_+^2(\gamma_{\mathbb S^{n-1}})_{AB}
\end{pmatrix}.
\]
The angular block is positive definite.  Its Schur complement is
\[
2V_r
-
\frac1{r_+^2}|d_\omega V|_{\gamma_{\mathbb S^{n-1}}}^2
=
B(r_+,\omega).
\]
Consequently, the slice is spacelike at the future horizon precisely
when
\begin{equation}\label{eq:horizon-spacelike-B}
B(r_+,\omega)>0.
\end{equation}

Since
\[
F(r_+)=0,
\qquad
F'(r_+)=\frac{n-2}{r_+}>0,
\]
equations \eqref{eq:u-general-schwarzschild} and
\eqref{eq:horizon-spacelike-B} imply
\[
u(r_+,\omega)=0
\]
and
\begin{equation}\label{eq:schwarzschild-simple-zero}
\partial_r u(r_+,\omega)
=
F'(r_+)B(r_+,\omega)>0.
\end{equation}
\begin{revision}
To verify Proposition~\ref{prop:simplezero}, the defining function must
follow the presentation foliation, whose leaves are the level sets of
$f=V-r^*$ rather than, in general, the areal spheres.  Set
\[
c=F'(r_+)>0,
\qquad
r^*=c^{-1}\log(r-r_+)+h(r),
\]
where $h$ is smooth near $r_+$, and define
\begin{equation}\label{eq:schwarzschild-foliation-rho}
\rho=e^{-cf}=(r-r_+)e^{c h(r)-cV(r,\omega)}.
\end{equation}
Then $\rho$ is a smooth boundary defining function with the same leaves
as $f$, and $d\rho$ points toward the exterior at the boundary.  Choose
\[
\nu=\frac{d\rho}{|d\rho|_g},
\qquad
b=-\frac{\sqrt F\,|d\rho|_g}{c\rho}.
\]
This signed choice satisfies $b\nu=\sqrt F\,df$.  Furthermore,
\[
\frac{u}{\rho}
=
\frac{F(r)}{r-r_+}\,B(r,\omega)
 e^{-c h(r)+cV(r,\omega)}
\]
extends smoothly and positively to the boundary.  Transporting angular
coordinates by $\nabla_g\rho/|d\rho|_g^2$ gives an orthogonal collar
\[
g=\lambda^{-2}d\rho^2+\gamma_\rho,
\qquad \lambda=|d\rho|_g>0.
\]
Thus $u=0$ and $\partial_\rho u>0$ at $\rho=0$ in the actual foliation
coordinates, as required by Proposition~\ref{prop:simplezero}.
\end{revision}

The boundary cut
\[
\Sigma_0
=
\{r=r_+,\ v=V(r_+,\omega)\}
\]
is a MOTS.  Indeed, the vector field $\partial_v$ is tangent to the
future horizon, null there, and orthogonal to $\Sigma_0$.  The induced
metric on the horizon sections is
\[
r_+^2\gamma_{\mathbb S^{n-1}},
\]
which is independent of $v$.  The null second fundamental form in the
$\partial_v$ direction therefore has zero trace. In particular,
\[
\theta_+=0.
\]
Moreover,
\begin{equation}\label{eq:schwarzschild-horizon-area}
|\Sigma_0|_g
=
\omega_{n-1}r_+^{n-1}.
\end{equation}

We next record an explicit infinite-dimensional family satisfying the
remaining hypotheses of Theorem~\ref{thm:main}.
\begin{revision}
Choose $a>0$, $R_2>r_+$, and a smooth function $p$ on $[r_+,\infty)$
satisfying
\[
p(r)=aF(r)\quad\text{near }r_+,
\qquad 0<p(r)<1\quad(r_+<r<R_2),
\qquad p(r)=1\quad(r\ge R_2).
\]
\end{revision}  Define a radial reference slice
$V_0$ by
\begin{equation}\label{eq:V0-def}
V_0'(r)=\frac{p(r)}{F(r)}.
\end{equation}
Near the horizon, $V_0'=a$, so $V_0$ extends smoothly in
Eddington--Finkelstein coordinates.  For large $r$,
\[
V_0'=(r^*)',
\]
and hence
\[
V_0-r^*
\]
is constant.  Thus the corresponding spacelike slice agrees with a
Schwarzschild $t=\mathrm{const}$ slice near infinity.
For the reference slice,
\begin{equation}\label{eq:B0}
B_0
=
2V_0'-FV_0'^2
=
\frac{p(2-p)}{F}>0.
\end{equation}
In particular,
\[
B_0(r_+)=2a.
\]

Now let
\[
\eta\in
C^\infty\bigl([r_+,\infty)\times\mathbb S^{n-1}\bigr)
\]
\rev{have support contained in $[r_+,R_1]\times\mathbb S^{n-1}$
for some $r_+<R_1<R_2$, and define}
\begin{equation}\label{eq:Vepsilon}
V_\varepsilon
=
V_0+\varepsilon\eta.
\end{equation}
Then
\begin{align}
B_\varepsilon
={}&
B_0
+
2\varepsilon(1-FV_0')\eta_r
-
F\varepsilon^2\eta_r^2
-
\frac{\varepsilon^2}{r^2}
|d_\omega\eta|^2.
\label{eq:Bepsilon}
\end{align}
Since $B_0>0$, it follows that for sufficiently small
$|\varepsilon|$,
\[
B_\varepsilon>0.
\] Hence $V_\varepsilon$ determines a smooth spacelike slice crossing the
future horizon. If
\[
\eta(r_+,\omega)
\]
\rev{is nonconstant, then the slice intersects the horizon in a cut
which is not invariant under the standard spherical action in the
spacetime, namely}
\[
v
=
V_0(r_+)+\varepsilon\eta(r_+,\omega).
\]
\rev{The induced metric on this cut remains the round metric
$r_+^2\gamma_{\mathbb S^{n-1}}$.  Its angular dependence describes the
cut in spacetime and the surrounding initial data, rather than a
nonround intrinsic horizon geometry.  The family
\eqref{eq:Vepsilon} thus contains arbitrary sufficiently small angular
perturbations with the prescribed support.} Near the horizon, the radial reference metric is
\[
\rev{g_0}
=
\left(2a-a^2F\right)dr^2
+
r^2\gamma_{\mathbb S^{n-1}},
\]
and therefore
\[
\rev{H_{\Sigma_0}(g_0)}
=
\frac{n-1}{r_+\sqrt{2a}}>0.
\]
Consequently, for sufficiently small $|\varepsilon|$, continuity gives
\[
\rev{H_{\Sigma_0}(g)>0}
\]
for the nonspherical slices as well. 

\begin{revision}
The presentation extends globally through the asymptotic region.
Indeed,
\[
f_0'=\frac{p-1}{F}<0\qquad(r_+<r<R_2).
\]
On the support of $\eta$ this derivative is bounded away from zero
in absolute value (and tends to $-\infty$ at the horizon), so for
sufficiently small $|\varepsilon|$ the perturbed $f_r$ is still
negative there.  Thus $df$ has no new zeros.  On $r<R_2$, use
\[
\nu=-\frac{df}{|df|_g},\qquad b=-\sqrt F\,|df|_g;
\]
these agree with the preceding horizon presentation.  Outside the
support of $\eta$ the foliation is radial, and the same fields are
\[
\nu=\frac{dr}{|dr|_g},\qquad b=\sqrt F\,f_0'|dr|_g.
\]
The latter expressions extend smoothly across $r=R_2$, with $b=0$
for $r\ge R_2$.  The remaining presentation parameters are obtained
from $k$ as in Proposition~\ref{prop:static-graphs}; they extend by
$A=Q=\hat k=0$ on the constant-graph region.  The slice is therefore
exactly a Schwarzschild $t=\mathrm{const}$ slice near infinity.  Hence
\end{revision}
\[
k=0
\]
near infinity and
\[
E_{\rm ADM}(g)=E_{\rm ADM}(\bar g)=m,
\qquad
P_{\rm ADM}=0.
\]
The data are vacuum and therefore satisfy the dominant energy
condition.  We have thus obtained an infinite-dimensional family of
nonspherically symmetric maximally parameterized Schwarzschild slices
satisfying the hypotheses of Theorem~\ref{thm:main}.

\begin{revision}
These slices also have an outer-minimizing horizon.  For the radial
reference metric $g_0=B_0\,dr^2+r^2\gamma_{\mathbb S^{n-1}}$, the
areal spheres have outward mean curvature
\[
H_r(g_0)=\frac{n-1}{r\sqrt{B_0}}>0.
\]
This remains positive for all areal spheres under sufficiently small
$C^2$ perturbations $\varepsilon\eta$: on the compact perturbation
region it follows by continuity and a positive lower bound for
$H_r(g_0)$, and outside that region the metric is unchanged. Thus, the areal spheres form a mean convex foliation of the exterior with respect to $g$, which implies that \(\Sigma_0\) is outer minimizing by the divergence theorem applied to their outward unit normal. We then have
\end{revision}
\[
\mathcal A_g(\Sigma_0)=|\Sigma_0|_g
=
\omega_{n-1}r_+^{n-1}.
\]
\rev{Since}
\[
r_+^{n-2}=2m,
\]
\rev{these nonsymmetric slices attain equality:}
\[
m
=
\frac12
\left(
\frac{|\Sigma_0|_g}{\omega_{n-1}}
\right)^{\frac{n-2}{n-1}}.
\]

This example also explains geometrically the horizon hypothesis used
in the main theorem.  For smooth spacelike slices crossing the
nondegenerate Schwarzschild horizon, the simple zero of $u$ is simply
the consequence of the spacelike condition together with the simple
\rev{zero of $F=\phi^2$, the squared Schwarzschild lapse}.

A similar construction can be applied to find maximally parameterized slices of Reissner--Nordström without symmetry assumptions.

\subsection{Rotating data: Boyer--Lindquist slices of Kerr}

The maximally parameterized class also contains natural rotating
examples.  Consider the Kerr spacetime in Boyer--Lindquist coordinates
and write the stationary metric in ADM form
\begin{equation}\label{eq:kerr-ADM}
\widetilde g
=
-N^2dt^2
+
g_{ij}(dx^i+\rev{\zeta}^i dt)(dx^j+\rev{\zeta}^jdt).
\end{equation}
On a $t=\mathrm{const}$ slice, the shift is purely azimuthal:
\[
\rev{\zeta}
=
\rev{\zeta}^\varphi(r,\theta)\partial_\varphi.
\]
The induced metric $g$ and all of its coefficients are independent of
$\varphi$.

Let
\[
\Sigma_r=\{r=\mathrm{const}\},
\qquad
\nu=\frac{dr}{|dr|_g}.
\]
The shift $\rev{\zeta}$ is tangent to $\Sigma_r$.  Since the spacetime is
stationary, the second fundamental form of the $t=\mathrm{const}$
slice is, up to the overall sign convention for $k$,
\[
k_{ij}
=
\frac{1}{2N}
\left(
D_i\rev{\zeta}_j+D_j\rev{\zeta}_i
\right).
\]
Taking the tangential trace along $\Sigma_r$ gives
\[
\operatorname{tr}_{\Sigma_r}k
=
\frac{1}{N}
\operatorname{div}_{\Sigma_r}\rev{\zeta}
\]
up to the same overall sign.  Since $\rev{\zeta}$ has only a
$\varphi$-component and every coefficient is independent of
$\varphi$,
\[
\operatorname{div}_{\Sigma_r}\rev{\zeta}
=
\frac{1}{\sqrt{\det\gamma}}
\partial_\varphi
\left(
\sqrt{\det\gamma}\,\rev{\zeta}^\varphi
\right)
=
0.
\]
Therefore
\begin{equation}\label{eq:kerr-leaf-trace}
\operatorname{tr}_{\Sigma_r}k=0.
\end{equation}

Set
\[
b=0.
\]
Then $W=1$ and the trace equation becomes
\[
\rev{\mathrm{I}}
=
-\operatorname{tr}_{\Sigma_r}k
=
0.
\]
Moreover,
\[
d_\Sigma b=0,
\]
so the zero-divergence equation $\rev{\mathrm{II}}=0$ is automatic.

More explicitly, define
\[
A=k(\nu,\nu),
\qquad
Q=-k(\nu,\cdot)|_{\Sigma_r},
\]
and let
\[
\hat k
=
k|_{T\Sigma_r}.
\]
Equation \eqref{eq:kerr-leaf-trace} implies that $\hat k$ is
tangentially trace-free.  Substituting into \eqref{eq:kMP} gives
exactly the Kerr second fundamental form.  Thus the standard
Boyer--Lindquist Kerr initial data are maximally parameterized.

For nonzero angular momentum, the mixed normal--tangential component
of $k$ is nonzero, and hence
\[
Q\neq0
\]
in general.  The Kerr example therefore gives a natural rotating
example in which the $Q$ degrees of freedom in the maximally
parameterized presentation are genuinely present.

Here
\[
\bar g=g
\]
because $b=0$.  This example is intended to illustrate the size of the
maximally parameterized class rather than the simple-zero horizon
mechanism in Theorem~\ref{thm:main}; the presentation above does not
produce $u\to0$ at the Kerr horizon.

\section{AI statement}

The \rev{gauge-invariant Schoen--Yau identity ($t=1$)} was derived
entirely by humans in early February 2026. By mid-February, we began
using Claude to help with the algebra (which is quite messy) leading
to the \rev{new Schoen--Yau identity for general $t$}.
We used Claude to verify every identity numerically as well as
algebraically, thereby providing a double-check against typos, sign
errors, or other mistakes.  In addition to being an outstanding
assistant, Claude also made some useful mathematical suggestions (as
we derived the special form for the second fundamental form $k$
\rev{in \eqref{eq:kMP}}) which \rev{warrant} a sincere
acknowledgement at a minimum. Had a human researcher made similar
mathematical suggestions, we would have invited them to be a
co-author. A younger researcher would likely have accepted, and a
senior human researcher might have declined the offer. The help was
appreciated, but not decisive.
 
We also used Claude to draft ideas in short notes which we shared
with each other and friends for discussion purposes. However, this
paper was written by humans, using the AI drafted notes as a
reference. \rev{Our Claude conversations were conducted under
settings that exclude them from model training, and we have felt in
control of our user experience throughout.}
 
We also played around with Claude looking for an existence theory to
various regularized versions of our equations, which in a best case
would result in a full proof of the Penrose conjecture. The
resulting discussions got very complicated, resulting in us having
less confidence in their correctness. Even if they were correct,
Claude would have been the lead author, and we didn't want to put
out a paper whose ideas were not optimized both for correctness and
presentation.  We think this is a great direction for the field and
warrants further study by others. We decided the present paper was
interesting enough for now.
 
We are also impressed by the productivity of OpenAI's series of
papers \rev{\cite{seriesofpapers}} on the Penrose conjecture and
related topics, but \rev{the mode of release concerns us}. This
massive \rev{drop of papers} is both good for the field in some
ways (new ideas) and bad for the field in other ways (mistakes,
confusing exposition, lack of clarity what to do next, massive
checking needing to be done, etc.).
 
For example, we think the present paper clarifies what appears to be
the overall direction in the OpenAI paper \rev{\cite{OpenAIPenrose}}
on the proof of the Penrose conjecture. Our new equations
\rev{\eqref{eq:gauge-system}} are degenerate and so need to be
regularized. The OpenAI paper begins with \rev{what we recognize as
regularized versions} of our equations with very little motivation,
a bit like a magic trick. This matters for understanding (both human
and AI) and because it affects what we should do next. For example,
it is still open if our original degenerate system
\rev{\eqref{eq:gauge-system}} of equations has a general weak
existence theory. Providing this theory would be nice progress which
would make it easier for others to find even more applications.
 
As completion of this paper neared, we knew that leaving open
conjectures that imply the Penrose conjecture would inspire many
others to try to finish our program. But we were \rev{surprised}
when an OpenAI paper claiming to finish our program appeared
\rev{on Tuesday, October 6, 2026 (the paper itself carries an
earlier internal date),} the week {\it before} we posted.
\rev{Under current industry practices, privately shared research
can enter AI training data without the authors' knowledge --- for
example, if any recipient of a draft pastes it into a consumer AI
product whose default settings permit training. Authors have no way
to know whether this has happened to their work, and no lab can
demonstrate that it has not; we regard this as a serious and
growing problem for scholarly credit. We do not know whether our
privately circulated notes influenced the OpenAI results, and under
current practices we cannot know. Labs could resolve such questions
by publishing training-data provenance; absent that, authors are
left with timelines, so we state ours.} To be clear, we claim
credit for our program as defined in this paper, which we've been
developing since well before February 2026, the date of our first new key result.

\enlargethispage{2\baselineskip}
\begingroup
\showchangesfalse
\color{black}

\endgroup

\end{document}